\documentclass[12pt]{amsart}
\usepackage{tikz-cd}

\makeatletter

\theoremstyle{plain}

\newtheorem{thm}{Theorem}[section]

\newtheorem{cor}[thm]{Corollary} 

\newtheorem{lem}[thm]{Lemma} 
\newtheorem{prop}[thm]{Proposition} 
\theoremstyle{definition}
\newtheorem{defn}[thm]{Definition}
\theoremstyle{remark}
\newtheorem{rem}[thm]{Remark}
\theoremstyle{remark}

\theoremstyle{remark}

\theoremstyle{remark}

\theoremstyle{definition}

\theoremstyle{definition}

\theoremstyle{plain}

\theoremstyle{definition}

\theoremstyle{remark}

\theoremstyle{remark}

\theoremstyle{definition}

\theoremstyle{remark}
  \newtheorem*{acknowledgement*}{Acknowledgement}

\newcommand{\bigslant}[2]{{\raisebox{.2em}{$#1$}\left/\raisebox{-.2em}{$#2$}\right.}}

\newcommand{\R}{{\mathbb R}}

\newcommand{\N}{{\mathbb N}}
\newcommand{\Z}{{\mathbb Z}}
\newcommand{\Q}{{\mathbb Q}}

\newcommand{\hh}{{\mathcal H}}

\newcommand{\id}{\mathrm{id}}
\DeclareMathOperator{\tr}{tr}

\newcommand{\ip}[1]{\langle#1\rangle}

\DeclareMathOperator{\Aut}{Aut}

\newcommand{\T}{\mathbb{T}}

\DeclareMathOperator{\diag}{diag}

\def\freeprod{\font\bigsymbolsfont=cmsy10 scaled \magstep3
 \setbox0=\hbox{\bigsymbolsfont\char'003 }\mathord{\lower1pt\box0}}\relax\ignorespaces

\newcommand{\Hawaii}{Hawai\kern.05em`\kern.05em\relax i}

\usepackage{amssymb} \usepackage{amscd} \usepackage{hyperref}  \usepackage{tikz} \usepackage{soul}  \usetikzlibrary{matrix,arrows}

\makeatother

\begin{document}

\title[Crossed products of irrational rotation algebras by cyclic groups]{Isomorphism and Morita equivalence classes for crossed products of irrational rotation algebras by cyclic subgroups of $SL_2(\Z)$} 


\author{Christian B\"onicke}
\address{Mathematisches Institut der WWU M\"unster,
	\newline Einsteinstrasse 62, 48149 M\"unster, Deutschland}
\email{christian.boenicke@uni-muenster.de}

\author{Sayan Chakraborty}
\address{Mathematisches Institut der WWU M\"unster,
	\newline Einsteinstrasse 62, 48149 M\"unster, Deutschland}
\email{sayan.c@uni-muenster.de}

\author{Zhuofeng He}
\address{Graduate School of Mathematical Sciences, The University of Tokyo,
	\newline 3-8-1 Komaba Meguro-ku Tokyo 153-8914, Japan}
\email{hzf@ms.u-tokyo.ac.jp}

\author{Hung-Chang Liao}
\address{Mathematisches Institut der WWU M\"unster,
	\newline Einsteinstrasse 62, 48149 M\"unster, Deutschland}
\email{liao@uni-muenster.de}


\date{Oct. 10, 2017}

\maketitle

\begin{abstract}
	Let $\theta, \theta'$ be irrational numbers and $A, B$ be matrices in $SL_2(\Z)$ of infinite order. We compute the $K$-theory of the crossed product $\mathcal{A}_\theta\rtimes_A \Z$ and show that $\mathcal{A}_{\theta} \rtimes_A\Z$ and $\mathcal{A}_{\theta'} \rtimes_B \Z$ are $*$-isomorphic  if and only if $\theta = \pm\theta' \pmod \Z$ and $I-A^{-1}$ is matrix equivalent to $I-B^{-1}$. Combining this result and an explicit construction of equivariant bimodules, we show that $\mathcal{A}_{\theta} \rtimes_A\Z$ and $\mathcal{A}_{\theta'} \rtimes_B \Z$ are Morita equivalent if and only if $\theta$ and $\theta'$ are in the same $GL_2(\Z)$ orbit and $I-A^{-1}$ is matrix equivalent to $I-B^{-1}$. Finally, we determine the Morita equivalence class of $\mathcal{A}_{\theta} \rtimes F$ for any finite subgroup $F$ of $SL_2(\Z)$.
\end{abstract}


\section{Introduction}
For a given irrational number $\theta$, let $\mathcal{A}_\theta$ denote the irrational rotation algebra, i.e., the universal $C^*$-algebra generated by unitaries $U_1$ and $U_2$ satsifying
$$
U_2U_1 = e^{2\pi i \theta}U_1 U_2.
$$
Watatani \cite{Wat} and Brenken \cite{Brenken} introduced an action of $SL_2(\Z)$ on $\mathcal{A}_{\theta}$ by sending a matrix
$$
A=
\begin{pmatrix}
  a&b\\
   c&d
\end{pmatrix}
$$
to the automorphism $\alpha_A$ of $\mathcal{A}_{\theta}$ defined by 
$$
\alpha_A(U_1) := e^{\pi i (ac) \theta}U_1^{a} U_2^{c},\;\;\;\;\;\; \alpha_A(U_2) := e^{\pi i (bd) \theta} U_1^{b} U_2^{d}.
$$
Let $A\in SL_2(\Z)$ be a matrix of infinite order, and consider the restriction of the above action to the (infinite cyclic) subgroup generated by $A$. For notational convenience we write the resulting crossed product by $\Z$ as $\mathcal{A}_\theta \rtimes_A \Z$.

Recall that by the work of Pimsner and Voiculescu \cite{PV80} and Rieffel \cite{Rie81} two irrational rotation algebras $\mathcal{A}_{\theta}$ and $\mathcal{A}_{\theta'}$ are isomorphic if and only if $\theta = \pm \theta' \pmod \Z$. Moreover, Rieffel showed in \cite{Rie81} that $\mathcal{A}_{\theta}$ and $\mathcal{A}_{\theta'}$ are Morita equivalent if and only if $\theta$ and $\theta'$ are in the same $GL_2(\Z)$ orbit, that is, $\theta' = \frac{a\theta + b}{c\theta + d}$ for some matrix $\begin{pmatrix}
a & b \\
c & d
\end{pmatrix}$ in $GL_2(\Z)$. In this paper we prove analogous results for the crossed product $\mathcal{A}_\theta \rtimes_A \Z$.  More precisely, we determine the isomorphism and Morita equivalence classes of these crossed products in terms of the angle $\theta$ and the matrix $A$ up to some canonical equivalence relations.

\begin{thm} \label{thm:1.1} [Theorem \ref{thm:main_iso}]\label{1}
	Let $\theta,\theta'$ be irrational numbers and $A,B\in SL_2(\Z)$ be matrices of infinite order. Then the following are equivalent:
	\begin{enumerate}
		\item $\mathcal{A}_{\theta} \rtimes_A \Z$ and $\mathcal{A}_{\theta'} \rtimes_B \Z$ are $*$-isomorphic;
		\item $\theta = \pm \theta' \pmod \Z$ and $P(I_2 - A^{-1})Q = I_2 - B^{-1}$ for some $P,Q$ in $GL_2(\Z)$.
	\end{enumerate}
\end{thm}

\begin{thm} \label{thm:1.2}[Theorem \ref{thm:main_morita}]
	Let $\theta, \theta'$ be irrational numbers and $A,B\in SL_2(\Z)$ be matrices of infinite order. Then the following are equivalent:
	\begin{enumerate}
		\item $\mathcal{A}_{\theta}\rtimes_A \Z$ and $\mathcal{A}_{\theta'} \rtimes_B \Z$ are Morita equivalent;
		\item $\theta$ and $\theta'$ are in the same $GL_2(\Z)$ orbit,  and $P(I-A^{-1})Q = I-B^{-1}$ for some $P,Q\in GL_2(\Z)$.
	\end{enumerate}
\end{thm}

Let $F$ be a finte subgroup of $SL_2(\Z)$, which is necessarily isomorphic to $\Z_k$ with $k=2,3,4$ or $6$. The crossed product $\mathcal{A}_{\theta}\rtimes_{\alpha}F$ and the fixed point algebra $\mathcal{A}_{\theta}^F$ have been studied by many authors, including  \cite{7,9,ELPW10,26,27,28,29,30,31,32,38,62,Wal00,73,74}. Our study is particularly motivated by the work \cite{ELPW10} of Echterhoff, L\"uck, Phillips, and Walters. They  showed that when $\theta$ is irrational, the crossed product $\mathcal{A}_{\theta}\rtimes_{\alpha}\Z_k$ is an AF algebra and the isomorphism class of $\mathcal{A}_{\theta}\rtimes_{\alpha}\Z_k$ is completely determined by $\theta$ and $k$. Theorem \ref{1} can be viewed as an analogue of this result. 

Let us briefly discuss the proof of Theorem \ref{1}. By \cite{ELPW10}, \cite{OP06b} and \cite{Lin07}, each such $\Z$-action on $\mathcal{A}_{\theta}$ has the tracial Rokhlin property, and the resulting crossed product is monotracial and classifiable. Therefore the main step is to compute the Elliott invariant.  We compute the $K$-theory using the Pimsner-Voiculescu exact sequence. To determine a set of explicit generators and their images under the induced map of the unique tracial state, we borrow ideas from \cite{Thesis}, where a similar computation was carried out for crossed products of $C(\T^2)$. The precise description of the Elliott invariant is given in the following theorem. See Section 2.1 for more on Smith normal forms, and Section 3 for definitions of the elements $p_\theta$ and $P_A$. 

\begin{thm} \label{thm:1.3} [Theorem \ref{thm:K}]
      Let $\theta$ be an irrational number and $A\in SL_2(\Z)$ be a matrix of infinite order. Write $\tau_A$ for the unique trace on $\mathcal{A}_{\theta}\rtimes_A \Z$. 	
	\begin{enumerate}
		\item If $\tr(A) = 2$ then $I_2-A^{-1}$ has the Smith normal form $\diag(h_1,0)$, and
		\begin{align*}
			K_0(\mathcal{A}_{\theta} \rtimes_A \Z) &\cong \Z\oplus \Z\oplus \Z, \\
			K_1(\mathcal{A}_{\theta} \rtimes_A \Z) &\cong \Z\oplus \Z \oplus \Z \oplus \Z_{h_1},\\
			(\tau_A)_*( K_0(\mathcal{A}_{\theta} \rtimes_A \Z) ) &= \Z + \theta\Z.
		\end{align*}
		A set of generators of $K_0(\mathcal{A}_{\theta} \rtimes_A \Z)$ is given by $[1]_0$, $i_*([p_\theta]_0)$, and $[P_A]_0$, and the images under the unique trace are $1$, $\theta$, and $1$, respectively. 
		\item If $\tr(A) \not\in \{0,\pm1,2\}$ then $I_2-A^{-1}$ has the Smith normal form $\diag(h_1,h_2)$, and
		\begin{align*}
			K_0(\mathcal{A}_{\theta} \rtimes_A \Z) &\cong \Z\oplus \Z, \\
			K_1(\mathcal{A}_{\theta} \rtimes_A \Z) &\cong \Z \oplus \Z \oplus \Z_{h_1} \oplus \Z_{h_2}, \\
			(\tau_A)_*( K_0(\mathcal{A}_{\theta} \rtimes_A \Z) ) &= \Z + \theta\Z.
		\end{align*}
	A set of generators of $K_0(\mathcal{A}_{\theta} \rtimes_A \Z)$ is given by $[1]_0$ and $i_*([p_\theta]_0)$, and the images under the unique trace are $1$ and $\theta$, respectively.
	\end{enumerate}
\end{thm}

Now we turn to the proof of Theorem \ref{thm:1.2}. Combining our computation of the $K$-theory and the proof of Rieffel's original Morita equivalence result in \cite{Rie81}, we can see that if the crossed products $\mathcal{A}_\theta \rtimes_A \Z$ and $\mathcal{A}_{\theta'}\rtimes_B \Z$ are Morita equivalent, then $\theta$ and $\theta'$ are in the same $GL_2(\Z)$ orbit and $P(I-A^{-1})Q = I-B^{-1}$ for some $P,Q\in GL_2(\Z)$. The main ingredient of the reverse implication is a construction of suitable actions on the $\mathcal{A}_\theta$-$\mathcal{A}_{\frac{1}{\theta}}$-imprimitivity bimodule $\mathcal{S}(\R)$. This is done by studying the so-called \emph{metaplectic operators} on $\mathcal{S}(\R)$ (see Section 4). A similar argument also allows us to completely determine the Morita equivalence classes for crossed products of the form $\mathcal{A}_\theta \rtimes_\alpha F$ for any finite subgroup $F$ of $SL_2(\Z)$.

This paper is structured as follows. In Section 2 we recall various background materials related to irrational rotation algebras and the $SL_2(\Z)$-action on them. We also include a short discussion of the Smith normal form of an integral matrix. In Section 3 we compute the $K$-theory of the crossed product $\mathcal{A}_\theta\rtimes_A \Z$ (Theorem \ref{thm:1.3}) and prove Theorem \ref{thm:1.1}. Section 4 is devoted to the proof of Theorem \ref{thm:1.2}, and finally in Section 5 we determine the Morita equivalence classes of $\mathcal{A}_\theta \rtimes_\alpha F$ for any finite subgroup $F$ of $SL_2(\Z)$.

\ \newline
{\bf Acknowledgments}. Z. He thanks the staff and members of
the Mathematics Institute at WWU M\"unster for their help and hospitality
during the extended visit in summer 2017, during which this research was undertaken. We would like to thank Dominic Enders for suggesting an alternative proof of Theorem \ref{thm:mortia_finite}. We also like to thank Siegfried Echterhoff for helpful and enlightening discussions on the subject. 

C. B\"onicke, S. Chakraborty, and H-C. Liao were supported by \textit{Deutsche Forschungsgemeinschaft} (SFB 878). Z. He was supported by the FMSP program at the Graduate School of Mathematical Sciences of the University of Tokyo, and partially supported by \textit{Deutsche Forschungsgemeinschaft} (SFB 878).


\section{Preliminaries}

\subsection{Matrix equivalence and Smith normal form of integral matrices}
Here we only give the definitions and theorems needed for the paper. For a more comprehensive treatment of Smith normal forms, we refer the reader to \cite[Chapter 2]{Int}. 

Let $M_n(\Z)$ denote the set of $n$ by $n$ matrices with integer entries, and let $GL_n(\Z)$ be the group of elements in $M_n(\Z)$ with determinant $\pm1$.
\begin{defn}
	Let $A$ and $B$ be matrices in $M_n(\Z)$. We say \emph{$A$ is matrix equivalent to $B$}, written as $A\sim_{eq}B$, if there exist $P$ and $Q$ in $GL_n(\Z)$ such that 
	$$
	PAQ=B.
	$$ 
\end{defn}
It is easy to see that $\sim_{eq}$ is an equivalence relation on $M_n(\Z)$.
\begin{thm}\cite[Theorem II.9]{Int}
	Every matrix $A\in M_n(\Z)$ is matrix equivalent to a diagonal matrix 
	$$
	S(A)=\diag(h_1,h_2,\ldots,h_r,0,0,\ldots,0),
	$$
	where $r$ is the rank of $A$ and $h_1,h_2,\ldots,h_r$ are positive nonzero integers. 
	
	Moreover, the matrix $S(A)$ is unique subject to the condition that  $h_i$ divides $h_{i+1}$ for each $i=1,2,..., r-1$. In this case, $S(A)$ is called the \emph{Smith normal form} of $A$. 
\end{thm}

\begin{rem}
	Note that if $A$ has full rank and if the Smith normal form of $A$ is given by $\diag(h_1,h_2,\ldots,h_r)$, then we have $\vert \det(A)\vert=\prod_{i=1}^rh_i$.
\end{rem}

\subsection{Irrational rotation algebras}

Let $\theta\in \R\setminus \Q$.	The \emph{irrational rotation algebra}, or \emph{noncommutative 2-torus}, denoted by $\mathcal{A}_{\theta}$, is defined to be the universal $C^*$-algebra generated by two unitaries $U_1$ and $U_2$ satisfying the relation
$$
U_2U_1 = e^{2\pi i \theta}U_1 U_2.
$$

Recall that $\mathcal{A}_{\theta}$ is a unital, simple, separable, nuclear, monotracial $C^*$-algebra (see for example \cite{Dav}). We will write $\tau_\theta$ for the unique tracial state on $\mathcal{A}_{\theta}$. Since $\mathcal{A}_{\theta}$ is a simple AT $C^*$-algebra with real rank zero \cite{ElliottEvan}, by \cite{TAF} it has tracial rank zero in the sense of Lin. The $K$-theory of $\mathcal{A}_{\theta}$ was computed by Pimsner and Voiculescu in \cite{PV}, where they developed the celebrated Pimsner-Voiculescu six-term exact sequence for crossed products by integers. Here we summarize the $K$-theoretic data of $\mathcal{A}_{\theta}$:
\begin{itemize}
	\item $K_0(\mathcal{A}_{\theta}) \cong \Z^2$ with generators $[1]_0$ and $[p_\theta]_0$, where $p_\theta$ is a projection in $\mathcal{A}_{\theta}$ satisfying $\tau_\theta(p_\theta) = \theta$ (this is the so-called \emph{Rieffel projection}; see \cite[Theorem 1]{Rie81}).
	\item $K_1(\mathcal{A}_{\theta})\cong \Z^2$ with generators $[U_1]_1$ and $[U_2]_1$.
	\item $(\tau_\theta)_*:K_0(\mathcal{A}_{\theta})\to \R$ is an order isomorphism onto $\Z+\theta\Z$.
\end{itemize}
Note that the last item implies, among other things, that any automorphism $\alpha$ on $\mathcal{A}_{\theta}$ induces the identity map at the level of $K_0(\mathcal{A}_{\theta})$ (since there is no nontrivial unit-preserving order automorphism of $\Z+\theta\Z$).

\subsection{Twisted group $C^*$-algebras}

Let us consider $\mathcal{A}_{\theta}$ as a twisted group $C^*$-algebra. This picture of $\mathcal{A}_\theta$ will be useful when we discuss Morita equivalence classes of crossed products. For the following discussion we restrict our attention to discrete groups. Recall that a \emph{2-cocycle} on a discrete group $G$ is a function $\omega:G\times G\to \T$
satisfying
$$
\omega(x,y)\omega(xy,z) = \omega(x, yz)\omega(y,z)
$$
and
$$
\omega(x,1) = 1 = \omega(1,x)
$$
for all $x,y,z$ in $G$. We equip the Banach space $\ell^1(G)$ with the multiplication
$$
(f *_\omega g)(x) := \sum_{y\in G} f(y) g(y^{-1}x)\omega(y, y^{-1}x)\;\;\;\;\;\; (f,g\in \ell^1(G),\;\; x\in G)
$$
and the involution
$$
f^*(x) := \overline{ \omega(x,x^{-1})f(x^{-1}) }\;\;\;\;\;\; (f\in \ell^1(G).\;\; x\in G). 
$$
Then $\ell^1(G)$ becomes a Banach $*$-algebra. For clarity we write $\ell^1(G,\omega)$ for the resulting Banach *-algebra. As in the case of group $C^*$-algebras, we complete $\ell^1(G,\omega)$ with respect to the norm coming from the ``regular representation''. Recall that for a given 2-cocycle $\omega$ on $G$, an \emph{$\omega$-representation} of $G$ on a Hilbert space $\hh$ is a map
	$$
	V:G\to \mathcal{U}(\hh)
	$$
	satisfying
	$$
	V(x)V(y) = \omega(x,y)V(xy)
	$$
	for all $x,y\in G$. Note that every $\omega$-representation $V:G\to \mathcal{U}(\hh)$ can be promoted to a $*$-homomorphism $V:\ell^1(G,\omega)\to B(\hh)$ via the formula
$$
V(f) := \sum_{x\in G} f(x)V(x).
$$
Now consider the map $L_\omega:G\to \mathcal{U}(\ell^2(G))$ defined by
$$
[L_\omega(x)\xi ](y) := \omega(x,x^{-1}y) \xi(x^{-1}y)\;\;\;\;\;\; (\xi\in \ell^2(G),\;\; x,y\in G).
$$
A direct computation shows that $L_\omega$ is an $\omega$-representation of $G$, called the \emph{regular $\omega$-representation}. The \emph{reduced twisted group $C^*$-algebra}, written as $C^*_r(G,\omega)$, is defined to be the completion of $\ell^1(G,\omega)$ with respect to the norm $\| f \|_r := \| L_\omega(f) \|$.

Let $\theta$ be an irrational number. We may identify $\theta$ with the real $2\times 2$ skew-symmetric matrix $\begin{pmatrix}
0 & -\theta \\
\theta & 0
\end{pmatrix}$. Define a 2-cocycle $\omega_{\theta}:\Z^2\times \Z^2 \to \T$ by
$$
\omega_{\theta}(x,y) := e^{-\pi i \ip{ \theta x, y }}.
$$
Then there is a $*$-isomorphism $\mathcal{A}_{\theta}\to C^*_r(\Z^2,\omega_{\theta})$ sending $U_1$ to $\delta_{e_1}$ and $U_2$ to $\delta_{e_2}$, where $\{e_1,e_2\}$ is the standard basis for $\Z^2$.

\subsection{Actions of $SL_2(\Z)$ on irrational rotation algebras}

Let $SL_2(\Z)$ be the group of $2\times 2$ integer-valued matrices with determinant 1. For each matrix $A = \begin{pmatrix}
a & b \\
c & d
\end{pmatrix}$ in $SL_2(\Z)$, define an automorphism $\alpha_A:\mathcal{A}_{\theta}\to \mathcal{A}_{\theta}$ by declaring
$$
\alpha_A(U_1) := e^{\pi i (ac) \theta}U_1^a U_2^c,\;\;\;\;\;\; \alpha_A(U_2) := e^{\pi i (bd) \theta} U_1^b U_2^d.
$$
Note that the commutation relation holds because $A$ has determinant 1, and the scalars are there to ensure that the map $\alpha : SL_2(\Z) \to \Aut(\mathcal{A}_{\theta})$ sending $A$ to $\alpha_A$ is indeed a group homomorphism.

For each $A\in SL_2(\Z)$ we consider the $\Z$-action on $\mathcal{A}_{\theta}$ generated by $\alpha_A$. Throughout the paper, for the ease of notation we write $\mathcal{A}_{\theta} \rtimes_A\Z$ for the resulting crossed product.

\begin{lem} \label{lem:tracial-Rok}
	Let $\theta\in \R\setminus \Q$ and $A\in SL_2(\Z)$ be a matrix of infinite order. Then the $\Z$-action generated by $\alpha_A$ has the tracial Rokhlin property.
\end{lem}
\begin{proof}
	By \cite[Lemma 5.10]{ELPW10}, the extension of the automorphism $(\alpha_A)^n$ $(n\neq 0)$ to the weak closure of $\mathcal{A}_{\theta}$ in the tracial GNS representation is outer. Since $\mathcal{A}_{\theta}$ is a simple separable unital monotracial C*-algebra with tracial rank
	zero, the statement follows from \cite[Theorem 2.17]{OP06b}.
\end{proof}

\begin{thm} \label{thm:structure}
	Let $\theta\in \R\setminus \Q$ and $A\in SL_2(\Z)$ be a matrice of infinite order. Then $\mathcal{A}_{\theta}\rtimes_A\Z$ is a unital, simple, separable, nuclear, monotracial $C^*$-algebra with tracial rank zero and satisfies the UCT.
\end{thm}
\begin{proof}
It is well-known that every crossed product of a nuclear $C^*$-algebra by the integers is nuclear (see for example \cite[Theorem 4.2.4]{BrOz}). We have seen in the proof of Lemma \ref{lem:tracial-Rok} that the action generated by $\alpha_A$ is outer. Therefore simplicity follows from \cite[Theorem 3.1]{Kis}. The fact that $\mathcal{A}_\theta\rtimes_A\Z$ has tracial rank zero is a direct consequence of \cite[Theorem 3.16]{Lin07}.
Finally, by \cite[Proposition 2.7]{RoScho} the UCT class is closed under taking crossed products by integers. 
\end{proof}

\begin{rem}
For any matrix $A\in SL_2(\Z)$, the characteristic polynomial of $A$ is given by 
$$p(\lambda)=\lambda^2-\tr(A)\lambda+1.$$
It follows from the Cayley-Hamilton theorem that if $\tr(A)=\{0, \pm1\}$ then $A$ has finite order.
\end{rem}

In the present paper, when we consider $\Z$-actions we only allow matrices of infinite order. In particular we exclude the cases $\tr(A) \in \{0, \pm1\}$ and $A=\pm I_2$. 

Finally, in order to determine the Morita equivalence classes of these crossed products, it is important for us to understand how the action looks like in the twisted group $C^*$-algebra picture. This answer is given by the following proposition:

\begin{prop} \cite[p. 185]{ELPW10} \label{prop:action}
	Let $\alpha:SL_2(\Z)\curvearrowright \mathcal{A}_\theta$ be the canonical action. Then for any $A$ in $SL_2(\Z)$, $f\in \ell^1(\Z^2, \omega_{\theta})$, and $\ell \in \Z^2$, the action is given by
	$$
	(A.f)(\ell) := f( A^{-1} \ell ).
	$$
	In particular, if we write $U_\ell := \delta_\ell$ for $\ell\in \Z^2$, then $\alpha_A( U_\ell  ) = U_{A.\ell}$.
\end{prop}

In what follows we will use the notation $U_\ell$ as in Proposition \ref{prop:action}. However, we continue to use $U_1$ and $U_2$ for the canonical generators as in the introduction. In other words, we have 
$$
U_1 = U_{\begin{pmatrix}
1 \\ 0
\end{pmatrix} }\;\;\;\;\;\; \text{and}\;\;\;\;\;\; U_2 = U_{\begin{pmatrix}
	0 \\ 1
	\end{pmatrix} }.
$$

\subsection{Rieffel's Heisenberg equivalence bimodules}

We recall Rieffel's construction of Heisenberg equivalence bimodule, which connects $\mathcal{A}_{\theta}$ and $\mathcal{A}_{\frac{1}{\theta}}$. Again this is needed when we construct equivalence bimodules between crossed products and determine the Morita equivalence classes.  The exposition and the formulas below largely follow  \cite{ELPW10}. Throughout the paper we write $e(x) := e^{2\pi i x}$. Let $\mathcal{S}(\R)$ be the linear space consisting of all smooth and rapidly decreasing complex-valued functions on $\R$ (i.e., the \emph{Schwartz space}). Let $\mathcal{A}^\infty := \mathcal{S}(\Z^2,\omega_{\theta})$ be the dense *-subalgebra of $C^*_r(\Z^2,\omega_\theta)\cong \mathcal{A}_{\theta}$ consisting of all rapidly decreasing functions on $\Z^2$. To be more precise,
$$
\mathcal{S}(\Z^2, \omega_\theta) := \left\lbrace f\in \ell^1(\Z^2, \omega_\theta):\sup_{m,n\in \Z} (1+m^2+n^2)^k | f(m,n) | < \infty \text{ for all } k\in \N   \right\rbrace.
$$
Define a right-action of $\mathcal{A}^\infty$ on $\mathcal{S}(\R)$ by setting
$$
(f.U_1)(x) := f(x+\theta)\;\;\;\;\;\; \text{and} \;\;\;\;\;\;(f.U_2)(x) = e(x)f(x).
$$
Similarly let $\mathcal{B}^\infty := \mathcal{S}(\Z^2, \omega_{  \frac{1}{\theta} }  )$ and define a left-action of $\mathcal{B}^\infty$ on $\mathcal{S}(\R)$ by
$$
(V_1.f)(x) := f(x+1)\;\;\;\;\;\; \text{and} \;\;\;\;\;\; (V_2.f)(x) = e(-x/\theta  )f(x).
$$
It will be convenient to have a concrete formula for the actions of $U_\ell$ and $V_\ell$ on $\mathcal{S}(\R)$ for all $\ell\in \Z^2$.

\begin{prop} \label{prop:module_action}
	For all $f\in \mathcal{S}(\R)$ and all $\ell = \begin{pmatrix}
	m \\ n
	\end{pmatrix} \in \Z^2$, we have
	$$
	(f.U_\ell)(x) = e(mn\theta/2)e(nx)f(x+m\theta)
	$$
	and
	$$
	(V_\ell.f)(x) = e(-mn/(2\theta)) e(-nx/\theta)f(x+m)
	$$
\end{prop}
\begin{proof}
	We only prove the formula for $U_\ell$ as the other is completely analogous. Recall that we have
	$$
	U_{\ell} = \overline{ \omega_\theta \left( \begin{pmatrix}
		m \\ 0
		\end{pmatrix}, \begin{pmatrix}
		0 \\ n
		\end{pmatrix}   \right) } U_1^m U_2^n = e(mn\theta/2)U_1^mU_2^n.
	$$
	Therefore
	\begin{align*}
		(f.U_\ell)(x) &= e(mn\theta/2) (f.U_1^mU_2^n)(x) \\
		&= e(mn\theta/2)e(nx)(f.U_1^m)(x) \\
		&= e(mn\theta/2)e(nx)f(x+m\theta).
	\end{align*}
\end{proof}

Define $\mathcal{A}^\infty$-valued and $\mathcal{B}^\infty$-valued inner-products on  $\mathcal{S}(\R)$ by
$$
\ip{f,g}_{\mathcal{A}^\infty} (\ell) := \theta \int_\R \overline{ f(x+m\theta) } g(x) e(-nx) \;dx,
$$
$$
_{\mathcal{B}^\infty}\ip{f,g}(\ell) := \int_\R f(x-m) \overline{ g(x) } e(nx/\theta)\;dx,
$$
where $f,g\in \mathcal{S}(\R)$ and $\ell = \begin{pmatrix}
m \\ n
\end{pmatrix}\in \Z^2$.

\begin{thm}\cite[Theorem 1.1]{Rie83}, \cite[Theorem 2.15]{Rie88}, \cite[p. 646]{Wal00}
	With the actions and inner-products defined above, $\mathcal{S}(\R)$ becomes a $\mathcal{B}^\infty$-$\mathcal{A}^\infty$-pre-imprimitivity bimodule. In particular, $\mathcal{S}(\R)$ completes to a $\mathcal{A}_{\frac{1}{\theta}}$-$\mathcal{A}_{\theta}$-imprimitivity bimodule.
\end{thm}

Below we make an important observation that the inner-products on $\mathcal{S}(\R)$ can be realized using the actions and the usual $L^2$-inner product. This has been observed in \cite{CL17}.

\begin{prop} \label{prop:L2-trick}
	Let $f,g\in \mathcal{S}(\R)$ and $\ell = \begin{pmatrix}
	m \\ n
	\end{pmatrix}\in \Z^2$. Then 
	$$
	\ip{f,g}_{\mathcal{A}^\infty}(\ell) = \theta e(mn\theta/2) \ip{g.U_{-\ell}, f }_{L^2(\R)}
	$$
	and
	$$
	_{B_\infty}\ip{f,g} (\ell) = e(mn/(2\theta))\ip{f, V_\ell.g }_{L^2(\R)}.
	$$
\end{prop}
\begin{proof}
	Again we only verify the first equality.
	\begin{align*}
		\ip{g.U_{-\ell}, f}_{L^2(\R)} &= \int_\R (g.U_{-\ell})(x)\overline{f(x)}\;dx \\
		&= \int_\R e(mn\theta/2)e(-nx)g(x-m\theta)\overline{f(x)} \;dx \\
		&= e(mn\theta/2) \int_{\R} e(-n(x+m\theta)) g(x) \overline{  f(x+m\theta) } \;dx \\
		&= e(-mn\theta/2)\int_\R e(-nx) \overline{ f(x+m\theta) }g(x)\;dx.
	\end{align*}
	The proof is completed by comparing this expression with the $\mathcal{A}^\infty$-valued inner product $\ip{f,g}_{\mathcal{A}^\infty}(\ell)$.
\end{proof}


\section{The K-theory and isomorphism classes of $\mathcal{A}_{\theta}\rtimes_A \Z$}

In this section we compute the $K$-theory of crossed products of the form $\mathcal{A}_{\theta} \rtimes_A \Z$ and determine their isomorphism classes. Throughout the section $A$ will be a matrix in $SL_2(\Z)$ of infinite order. We also give a set of explicit generators for the $K_0$-group and compute their images under the induced map of the unique tracial state (recall from Theorem \ref{thm:structure} that the crossed product has a unique trace). 

The main tool is the Pimsner-Voiculescu sequence:
\[
\begin{CD}
K_0 (\mathcal{A}_{\theta}) @>{\id - \alpha^{-1}_{*0}}>>K_0 (\mathcal{A}_{\theta}))@>{i_{*}}>> K_0 (\mathcal{A}_{\theta} \rtimes_A \Z )   \\
@A{\delta_1}AA & &  @VV{\delta_0}V            \\
 K_1 (\mathcal{A}_{\theta}\rtimes_A \Z) @<<{i_{*}}<  K_1 (\mathcal{A}_{\theta}) @<<{\id - \alpha^{-1}_{*1}}< K_1 (\mathcal{A}_{\theta}).
\end{CD}
\]
Recall from Section 2.1 that the map $\id - \alpha^{-1}_{*0}$ is always the zero map. Therefore the sequence breaks into two short exact sequences 
\begin{equation*}
0 \longrightarrow \bigslant{ K_1(\mathcal{A}_{\theta})}{\mathrm{im}(\id-\alpha^{-1}_{*1})}\stackrel{i_*}{\longrightarrow} K_1(\mathcal{A}_{\theta}\rtimes_A\Z) \stackrel{\delta_1}{\longrightarrow} K_0(\mathcal{A}_{\theta}) \longrightarrow 0,
\end{equation*}
\begin{equation*}
0 \longrightarrow  K_0(\mathcal{A}_{\theta}) \stackrel {i_*}{\longrightarrow} K_0(\mathcal{A}_{\theta}\rtimes_A\Z)\stackrel{\delta_0}{\longrightarrow} \mathrm{ker}(\id-\alpha^{-1}_{*1}) \longrightarrow 0.
\end{equation*}
Note that each of the sequences splits because the second right-most group is free abelian.  Also note that since $K_1(\mathcal{A}_{\theta})$ is generated by $[U_1]_1$ and $[U_2]_1$, if we write $A  = \begin{pmatrix}
	a & b \\ c & d
	\end{pmatrix}$ then by the definition of the action we have 
$$
\id - \alpha^{-1}_{*1}([U_1]_1)=(1-d)[U_1]_1+b[U_2]_1,
$$
$$
\id - \alpha^{-1}_{*1}([U_2]_1)=c[U_1]_1+(1-a)[U_2]_1.
$$
Hence with respect to this set of generators the map $\id - \alpha^{-1}_{*1}$ is nothing but the matrix $I_2-A^{-1}$.

Let us begin by computing the $K_1$ group.

\begin{prop}
\begin{enumerate}
\item If $\tr(A)\not\in\{0,\pm1,2\}$ then $I_2-A^{-1}$ has the Smith normal form
$$
\begin{pmatrix}
  h_1&0\\
   0&h_2
\end{pmatrix}
$$
and the group $K_1(\mathcal{A}_{\theta}\rtimes_A\Z)$ is isomorphic to $\Z^2\oplus \Z_{h_1}\oplus\Z_{h_2}$.
\item If $\tr(A)=2$ then $I_2-A^{-1}$ has the Smith normal form
$$
\begin{pmatrix}
  h_1&0\\
   0&0
\end{pmatrix}
$$
and the group $K_1(\mathcal{A}_{\theta}\rtimes_A\Z)$ is isomorphic to $\Z^3\oplus \Z_{h_1}$.
\end{enumerate}
\end{prop}
\begin{proof}
We have seen that the sequence
\begin{equation*}
	0 \longrightarrow \bigslant{ K_1(\mathcal{A}_{\theta})}{\mathrm{im}(\id-\alpha^{-1}_{*1})}\stackrel{i_*}{\longrightarrow} K_1(\mathcal{A}_{\theta}\rtimes_A\Z) \stackrel{\delta_1}{\longrightarrow} K_0(\mathcal{A}_{\theta}) \longrightarrow 0
\end{equation*}
is split exact. Therefore we need only determine the quotient group	$ K_1(\mathcal{A}_{\theta})/\mathrm{im}(\id-\alpha^{-1}_{*1}) $. Since an element of $GL_2(\Z)$ only changes the basis of $K_1(\mathcal{A}_{\theta})=\Z[U_1]_1+\Z[U_2]_1$, the image of the Smith normal form of $I_2-A^{-1}$ is isomorphic to the image of $I_2-A^{-1}$. It remains to compute the Smith normal form (or the rank) of $I_2-A^{-1}$ in each case.

In general we have
$$
\det(I_2-A^{-1})=2-\tr(A).
$$
As a consequence the matrix $I_2-A^{-1}$ has full rank when $\tr(A)\not\in\{0,\pm1,2\}$, and we have completed the proof for case (1). When $\tr(A)=2$, one can show that the matrix $A$ is in the same conjugacy class in $GL_2(\Z)$ as
$$
A' =\begin{pmatrix}
  1&k\\
   0&1
\end{pmatrix}
$$
for some integer $k$ (see for example \cite[Lemma 5]{Phy}). Since $A$ has infinite order, $k$ must be nonzero. Observe that if $A$ and $A'$ are conjugate within $GL_2(\Z)$ then so are $I-A^{-1}$ and $I-(A')^{-1}$. Hence we conclude that $I-A^{-1}$ has rank one.
\end{proof}

We proceed to compute the $K_0$ group. We will first look at the case $\tr(A) \notin \{0,\pm 1, 2\}$ and then treat the case $\tr(A) = 2$, as in the latter finding a set of generators is more involved.

\begin{prop}
If $\tr(A)\not\in\{0,\pm1,2\}$, then the group $K_0(\mathcal{A}_{\theta}\rtimes_A \Z)$ is generated by $[1]_0$ and $[i(p_{\theta})]_0$.
\end{prop}
\begin{proof}
Consider the split short exact sequence
\begin{equation}\label{eq exact seq}
0 \longrightarrow  K_0(\mathcal{A}_{\theta}) \stackrel {i_*}{\longrightarrow} K_0(\mathcal{A}_{\theta}\rtimes_A\Z)\stackrel{\delta_0}{\longrightarrow} \mathrm{ker}(\id-\alpha^{-1}_{*1}) \longrightarrow 0.
\end{equation}
By assumpion $\tr(A)\neq\{0,\pm1,2\}$. It follows from the equation $\det(I_2-A^{-1})=2-\tr (A)$ that $I_2 - A^{-1}$ is full-rank. This implies that the kernel of $\id - \alpha^{-1}_{*1}$ is zero, and therefore we obtain an isomorphism
$$
i_*:K_0(\mathcal{A}_{\theta})\rightarrow K_0(\mathcal{A}_{\theta}\rtimes_A\Z),
$$
which maps generators to generators
\end{proof}

Now we consider the case $\tr(A)=2$ and $A\neq I_2$. We first study the special case that $A$ is equal to either
$$
\begin{pmatrix}
1&h_1\\
0&1
\end{pmatrix}\ \text{or\ } 
\begin{pmatrix}
1&0\\
h_1&1
\end{pmatrix} \;\;\;\;\;\; (h_1 \in \Z).
$$
Then we show that one can always reduce to these cases. Note that if $A$ has either one of the forms above, then the Smith normal form of $I_2 -A^{-1}$ is precisely $\diag(h_1, 0)$.

Let us assume that $A$ is equal to the former one (we will see that the other case is completely analogous). Then $U_1$ is a fixed point of the automorphism $\alpha_A$. Hence we have
$$
\mathrm{ker}(\id-\alpha^{-1}_{*1})=\Z[U_1]_1,
$$
and we need to find a preimage of $[U_1]_1$ under $\delta_0$. Recall that the Pimsner-Voiculescu sequence is natural since the six-term exact sequence for Toeplitz extension is natural (see \cite{PV}). Now consider the homomorphism
$$
\varphi:C(\T)\rightarrow \mathcal{A}_{\theta},\quad \iota\mapsto U_1,
$$
where $\iota:\T\rightarrow \mathbb{C}$ is the inclusion function, which generates $C(\T)$ as a $C^*$-algebra. Since $U_1$ is fixed by $\alpha_A$, we have the follwoing commutative diagram
$$
\begin{CD}
C(\T)@>\varphi>>\mathcal{A}_{\theta}\\
@VV= V@VV\alpha_A V\\
C(\T)@>\varphi>>\mathcal{A}_{\theta}.
\end{CD}
$$
Regard $C(\T^2)$ as $C(\T)\rtimes_{\tau}\Z$, where $\tau$ is the trivial action. Let $w'$ be the implementing unitary in $C(\T)\rtimes_{\tau}\Z$ coming from $\Z$. Also let $w$ be the implementing unitary in $\mathcal{A}_{\theta} \rtimes_A \Z$ of the action. The unitaries $U_1$ and $w$ generate a commutative $C^*$-subalgebra of $\mathcal{A}_{\theta}\rtimes_A\Z$ since $U_1$ is a fixed point. Thus there is a $*$-homormophism $\varphi\rtimes\id:C(\T^2)\rightarrow\mathcal{A}_{\theta}\rtimes_A\Z$ satisfying
$$
(\varphi\rtimes\id)\left(\sum_{n\in\Z}f_n(w')^n \right) = \sum_{n\in\Z}\varphi(f_n)w^n \;\;\;\;\;\; (f_n\in C(\T)).
$$
By naturality of the Pimsner-Voiculescu sequence, we have the following commutative diagram
$$
\begin{CD}
K_0(C(\T^2))@>(\varphi\rtimes\id)_*>>K_0(\mathcal{A}_{\theta}\rtimes_A\Z)\\
@VV\delta_0' V@VV\delta_0 V\\
K_1(C(\T))@>\varphi_*>>K_1(\mathcal{A}_{\theta}).
\end{CD}
$$

As in \cite{Thesis}, define $X_0$ and $X_1\in C([0,1])$ by
$$
X_0(t)=\frac{1}{4}\begin{pmatrix}
3+\cos(2\pi t)&\sin(2\pi t)\\
\sin(2\pi t)&1-\cos(2\pi t)
\end{pmatrix},
$$
$$
X_1(t)=\frac{1}{8}
\begin{pmatrix}
1-\cos(2\pi t)&\sqrt{2(1-\cos(2\pi t))}-\sin(2\pi t)\\
-\sqrt{2(1-\cos(2\pi t))}-\sin(2\pi t)&-1+\cos(2\pi t)
\end{pmatrix}.
$$
These  can be regarded as elements in $M_2(C(\T))$. Further define $P$ in $M_2(C(\T^2))$ by
$$
P(y,e^{2\pi i t})=
\begin{pmatrix}
\overline{y}&0\\
0&\overline{y}
\end{pmatrix}
X_1(t)^*+
X_0(t)+
X_1(t)
\begin{pmatrix}
y&0\\
0&y
\end{pmatrix}.
$$
It was shown in \cite{Thesis} that $P$ is a projection, $[1]_0$ and $[P]_0$ generate $K_0(C(\T^2))$, and
$$
\delta_0'([P]_0)=[\iota]_1.
$$
Hence if we write $P_{U_1,w} := (\varphi\rtimes\id)(P)$ then we have the following commutative diagram.
$$
\begin{CD}
[P]_0@>(\varphi\rtimes\id)_*>>[P_{U_1,w}]_0\\
@VV\delta_0' V@VV\delta_0 V\\
[\iota]_1@>\varphi_*>>[U_1]_1.
\end{CD}
$$
and we have found a preimage of $[U_1]_1$ under the index map $\delta_0$. 

Now if $A = \begin{pmatrix}
1 & 0 \\
h_1 & 1
\end{pmatrix}$, then $U_2$ is a fixed point of the action generated by $A$. We can repeat the entire discussion and obtain an element $P_{U_2,w}$, which represents a preimage of $[U_2]_1$.

\begin{prop}
	When $A$ is equal to either
	$$
	\begin{pmatrix}
	1&h_1\\
	0&1
	\end{pmatrix}\ \text{or\ } 
	\begin{pmatrix}
	1&0\\
	h_1&1
	\end{pmatrix},
	$$
	the group $K_0(\mathcal{A}_{\theta}\rtimes_A \Z)$ is generated by $[1]_0$, $[i(p_{\theta})]_0$, and $[P_{U_1,w}]_0$ (or $[P_{U_2,w}]_0$ in the latter case). 
	
	Moreover, under the unique tracial state of $\mathcal{A}_{\theta}\rtimes_A \Z$, the image of these generators are respectively $1$, $\theta$, and $1$.
\end{prop}
\begin{proof}
It remains to prove that the element $P_{U_1,w}$ (or $P_{U_2,w}$) has trace one. By Euler's formula we have 
	$$
	X_0(t)=\frac{1}{4}\begin{pmatrix}
	3+\frac{1}{2}(e^{2\pi i t}+e^{-2\pi i t})& \frac{1}{2i}(e^{2\pi i t}-e^{-2\pi i t})\\
	\frac{1}{2i}(e^{2\pi i t}-e^{-2\pi i t})&1-\frac{1}{2}(e^{2\pi i t}+e^{-2\pi i t})
	\end{pmatrix}
	$$
	and 
	$$
	X_1(t)=\frac{1}{8}
	\begin{pmatrix}
	1-\frac{1}{2}(e^{2\pi i t}+e^{-2\pi i t})&\sqrt{2-(e^{2\pi i t}+e^{-2\pi i t})}-\frac{1}{2i}(e^{2\pi i t}-e^{-2\pi i t})\\
	-\sqrt{2-(e^{2\pi i t}+e^{-2\pi i t})}-\frac{1}{2i}(e^{2\pi i t}-e^{-2\pi i t})&-1+\frac{1}{2}(e^{2\pi i t}+e^{-2\pi i t})
	\end{pmatrix}.
	$$
	Hence we can express $P_{U_1,w}$ in terms of $U_1$ and $w$ using functional calculus.
	Applying (the canonical extension of) the unique tracial state $\tau_A$ of $\mathcal{A}_{\theta}\rtimes_A\Z$ to $P_{U_1,w}$ we get
	$$
	\tau_A(P_{U_1,w})=\tau_A\left(
	\frac{1}{4}  \begin{pmatrix}
	3+\frac{1}{2}(w+w^*)&\frac{1}{2i}(w-w^*)\\
	\frac{1}{2i}(w-w^*)&1-\frac{1}{2}(w+w^*)
	\end{pmatrix}\right)
	=\tau_A(1)=1.$$
This completes the proof
\end{proof}

\begin{rem}
The construction of $X_0$, $X_1$, and $P$ in \cite{Thesis} was inspired by a proposition in the appendix of \cite{PV}, which says that if $\mathcal{A}$ is a unital $C^*$-algebra and $\alpha$ is an automorphism of $\mathcal{A}$, then given  a projection of the form $p=u^*x_1^*+x_0+x_1u\in\mathcal{A}\rtimes_{\alpha}\Z$ for some $x_0$ and $x_1\in\mathcal{A}$ the unitary $\exp(2\pi i x_0l_{x_1})$ is in $\mathcal{A}$ and satisfies 
	$$
	\delta_0([p]_0)=[\exp(2\pi i x_0l_{x_1})]_1.
	$$
Here $l_{x_1}$ denotes the left support projection of $x_1$ in the enveloping von Neumann algebra of $\mathcal{A}$.
\end{rem}

For a generic $A$ in $SL_2(\Z)$ with $\tr(A)=2$ and $A\neq I_2$, we know that the matrix $A$ is in the same conjugacy class in $SL_2(\Z)$ with either 
$$
\begin{pmatrix}
1&h_1\\
0&1
\end{pmatrix}\ \text{or\ } 
\begin{pmatrix}
1&0\\
h_1&1
\end{pmatrix}.
$$
This is because one $GL_2(\Z)$-conjugacy class breaks into at most two $SL_2(\Z)$-conjugacy classes and one can verify the above two matrices are not similar in $SL_2(\Z)$. Suppose we have
$$
QAQ^{-1}=\begin{pmatrix}
1&h_1\\
0&1
\end{pmatrix}=:B
$$
for some $Q\in SL_2(\Z)$. Then the automorphism of $\mathcal{A}_{\theta}$ determined by $Q$ intertwines the actions coming from $A$ and $B$. Hence we obtain an isomorphism between $\mathcal{A}_{\theta}\rtimes_A\Z$ and $\mathcal{A}_{\theta}\rtimes_B\Z$, which preserves the unique tracial states. Let $P_A$ denote the image of $P_{U_1,w}$ (or $P_{U_2,w}$) under this isomorphism.


\begin{thm}\label{thm:K}
	Let $\theta$ be an irrational number and $A\in SL_2(\Z)$ be a matrix of infinite order. Write $\tau_A$ for the unique trace on $\mathcal{A}_{\theta}\rtimes_A \Z$. 	
	\begin{enumerate}
		\item If $\tr(A) = 2$, then $I_2-A^{-1}$ is of rank $1$. Suppose $I_2-A^{-1}$ has the Smith normal form of $\diag(h_1,0)$, then
		\begin{align*}
			K_0(\mathcal{A}_{\theta} \rtimes_A \Z) &\cong \Z\oplus \Z\oplus \Z, \\
			K_1(\mathcal{A}_{\theta} \rtimes_A \Z) &\cong \Z\oplus \Z \oplus \Z \oplus \Z_{h_1},\\
			(\tau_A)_*( K_0(\mathcal{A}_{\theta} \rtimes_A \Z) ) &= \Z + \theta\Z.
		\end{align*}
		The generators of $K_0$ are given by $[1]_0$, $i_*([p_\theta]_0)$, and $[P_A]_0$, and the images under the unique trace are $1$, $\theta$, and $1$, respectively.
		\item If $\tr(A) \not\in\{0,\pm1,2\}$, then $I_2-A^{-1}$ is of rank $2$. Suppose $I_2-A^{-1}$ has the Smith normal form of $\diag(h_1,h_2)$, then
		
		\begin{align*}
			K_0(\mathcal{A}_{\theta} \rtimes_A \Z) &\cong \Z\oplus \Z, \\
			K_1(\mathcal{A}_{\theta} \rtimes_A \Z) &\cong \Z \oplus \Z \oplus \Z_{h_1} \oplus \Z_{h_2}, \\
			(\tau_A)_*( K_0(\mathcal{A}_{\theta} \rtimes_A \Z) ) &= \Z + \theta\Z.
		\end{align*}
		The generators of $K_0$ are given by $[1]_0$ and $i_*([p_\theta]_0)$, and the images under the unique trace are $1$ and $\theta$, respectively.
	\end{enumerate}
\end{thm}

\begin{cor}\label{cor:3}
Let $\theta$ be an irrational number and $A\in SL_2(\Z)$ be a matrix of infinite order. Then the crossed product $\mathcal{A}_{\theta}\rtimes_A\Z$ is an AH algebra with real rank zero and no dimension growth. Moreover, it is an AT algebra if and only if either $\tr(A)=3$ or $\tr(A) = 2$ and the greatest common divisor of the entries in $I_2-A^{-1}$ is one.
\end{cor}
\begin{proof}
We saw in Theorem \ref{thm:structure} that the crossed product is simple, nuclear, has tracial rank zero, and satisfies the UCT. Then it follows from \cite[Proposition 3.7]{everysimple} that the crossed product $\mathcal{A}_{\theta}\rtimes_A\Z$ is an AH algebra with real rank zero and no dimension growth. Furthermore, the same proposition shows that the crossed product is an AT algebra if and only if its $K$-groups are torsion-free. By Theorem \ref{thm:K}, $K_0(\mathcal{A}_\theta \rtimes _A \Z ) $ is torsion-free if and only if the Smith normal form of $I_2-A^{-1}$ is equal to either $\diag(1,0)$ or $\diag(1,1)$. The former corresponds to the case $\tr(A) = 2$ and $h_1 = 1$. The algorithm of computing the Smith normal form asserts that $h_1$ is precisely the greatest common divisor of the entries. On the other hand, if the Smith normal form of $I_2 -A^{-1}$ is equal to $\diag(1,1)$, then $|2-\tr(A)| = 1$. Since $A$ has infinite order, this happens if and only if $\tr(A) = 3$.
\end{proof}

To apply the classification theorem, it is necessary to describe the order structure of the $K_0$-group. It is shown in \cite{TAF} that a unital simple separable $C^*$-algebra with tracial rank zero has real rank zero, stable rank one, weakly unperforated $K_0$-group and is quasidiagonal. Since a weakly unperforated group is unperforated if and only if it is torsion-free, the $K_0(\mathcal{A}_{\theta}\rtimes_A \Z)$ for our cases is unperforated.

The following proposition is well-known to experts. We include a proof only for the reader's convenience.

\begin{prop} \label{prop:order-iso}
Let $\mathcal{A}$ and $\mathcal{B}$ be unital simple $C^*$-algebras with weakly unperforated $K_0$-groups and with unique tracial states $\tau_{\mathcal{A}}$ and $\tau_{\mathcal{B}}$ respectively. Then a group isomorphism 
$$
f:K_0(\mathcal{A})\rightarrow K_0(\mathcal{B})
$$
is an order isomorphism if $
(\tau_{\mathcal{B}})_*\circ f= k(\tau_{\mathcal{A}})_*$ for some $k>0$.
\end{prop}
\begin{proof}
Since $f$ is a group isomorphism, by symmetry it is sufficient to show that $f(K_0(\mathcal{A})^+)\subseteq K_0(\mathcal{B})^+$. For any nonzero projection $p\in\mathcal{P}_{\infty}(\mathcal{A})$, by assumption we have
$$
(\tau_{\mathcal{B}})_*\circ f ([p]_0)=k(\tau_{\mathcal{A}})_*([p]_0)= k\tau_\mathcal{A}(p) > 0
$$
(note that $\tau_{\mathcal{A}}$ is faithful since $\mathcal{A}$ is simple). It follows from Theorem 4.12 in \cite{Good} and Theorem 5.2.2 in \cite{Ror} (which is proved by \cite{BR92} and \cite{Haar}) that there exists a positive integer $n\in \N$ such that $nf([p]_0)$ is an order unit for $K_0(\mathcal{B})$. In particular $nf([p]_0)$ is nonzero. Since $K_0(\mathcal{B})$ is weakly unperforated, we conclude that $f([p]_0)$ belongs to $K_0(\mathcal{B})^+$ .
\end{proof}

Below we recall the definition of the Elliott invariant of a $C^*$-algebra. Here we restrict ourselves to the stably finite case.

\begin{defn} \label{defn:Ell} \cite[Definition 2.2.6]{Ror02}
	Let $\mathcal{A}$ be a unital simple separable $C^*$-algebra with nonempty trace space $T(\mathcal{A})$. The \emph{Elliott invariant} of $\mathcal{A}$, written as $Ell(\mathcal{A})$, is the 6-tuple
	$$
	( K_0(\mathcal{A}), K_0(\mathcal{A})^+, [1_\mathcal{A}]_0, K_1(\mathcal{A}), T(\mathcal{A}), r_\mathcal{A}:T(\mathcal{A})\to S(K_0(\mathcal{A})) ),
	$$
	where $S(K_0(\mathcal{A}))$ is the state space of $K_0(\mathcal{A})$. 
	
	We say two $C^*$-algebras $\mathcal{A}$ and $\mathcal{B}$ have \emph{isomorphic} Elliott invariants if there is an ordered group isomorphism $$
	\alpha_0:(K_0(\mathcal{A}), K_0(\mathcal{A})^+, [1_\mathcal{A}]_0)\to (K_0(\mathcal{B}), K_0(\mathcal{B})^+, [1_\mathcal{B}]_0),
	$$
	a group isomorphism $\alpha_1: K_1(\mathcal{A})\to K_1(\mathcal{B})$, and an affine homeomorphism $\gamma: T(\mathcal{B})\to T(\mathcal{A})$ such that $r_\mathcal{A}\circ \gamma = \widehat{\alpha}_0\circ r_\mathcal{B}$, where $\widehat{\alpha}_0:S(K_0(\mathcal{B}))\to S(K_0(\mathcal{A}))$ is defined by $\widehat{\alpha}_0(f) = f\circ \alpha_0$.
\end{defn}

\begin{thm} \label{thm:main_iso}
	Let $\theta,\theta'$ be irrational numbers and $A,B\in SL_2(\Z)$ be matrices of infinite order. Then the following are equivalent:
	\begin{enumerate}
		\item $\mathcal{A}_{\theta} \rtimes_A \Z$ and $\mathcal{A}_{\theta'} \rtimes_B \Z$ are $*$-isomorphic;
		\item $\mathrm{Ell}( \mathcal{A}_{\theta} \rtimes_A \Z )$ and $\mathrm{Ell}( \mathcal{A}_{\theta'} \rtimes_B \Z )$ are isomorphic;
		\item $\theta = \pm \theta' \pmod \Z$ and $I - A^{-1}\sim_{eq} I - B^{-1}$.
	\end{enumerate}
\end{thm}
\begin{proof}
	$(1)\Longrightarrow (2)$: This is obvious.\\
	$(2)\Longrightarrow (1)$: Since the crossed products satisfy the UCT, this follows from the classification theorem of tracially AF algebras by Lin \cite[Theorem 5.2]{Lin04}.\\
	$(2)\Longrightarrow (3)$: We get the condition on $\theta$ and $\theta'$ by comparing the images of $K_0$ under the (unique) tracial state. The condition on $A$ and $B$ follows from $K_1$.\\
	$(3)\Longrightarrow (2)$: Let us write $\mathcal{C} := \mathcal{A}_{\theta}\rtimes_A\Z$ and $\mathcal{D} := \mathcal{A}_{\theta'}\rtimes_B\Z$. By the previous proposition, we need only find a group isomorphism $f:K_0(\mathcal{C})\to  K_0(\mathcal{D})$ satisfying $
	(\tau_{\mathcal{D}})_*\circ f= (\tau_{\mathcal{C}})_*$ and $f([1_\mathcal{C}]_0) = [1_\mathcal{D}]_0$. Without loss of generality we may assume $\theta=1-\theta'$. Since $I - A^{-1}\sim_{eq} I - B^{-1}$, they have a common Smith normal form. In the case that $\tr(A)\notin \{0,\pm 1, 2\}$, define $f:K_0(\mathcal{C})\to  K_0(\mathcal{D})$ by
$$
f([i(1_{\mathcal{A}_{\theta}})]_0)=[i(1_{\mathcal{A}_{\theta'}})]_0,\quad f([i(p_{\theta})]_0)=[i(1_{\mathcal{A}_{\theta'}})]_0-[i(p_{\theta'})]_0.
$$
If $\tr(A) = 2$ then define $f:K_0(\mathcal{C})\to  K_0(\mathcal{D})$ by
$$
f([i(1_{\mathcal{A}_{\theta}})]_0)=[i(1_{\mathcal{A}_{\theta'}})]_0,\quad f([i(p_{\theta})]_0)=[i(1_{\mathcal{A}_{\theta'}})]_0-[i(p_{\theta'})]_0,\quad f([P_A]_0)=[P_B]_0.
$$
\end{proof}

\begin{rem}
We discuss an interesting corollary of Theorem \ref{thm:K}. Suppose $A\in SL_2(\Z)$ has infinite order and $\tr(A)=3$. By Corollary \ref{cor:3} the crossed product $\mathcal{A}_{\theta}\rtimes_A\Z$ is an AT algebra. Moreover, from Proposition \ref{prop:order-iso} and the proof of Theorem \ref{thm:K} we see that the map
$$
i_*:K_0(\mathcal{A}_{\theta})\rightarrow K_0(\mathcal{A}_{\theta}\rtimes_A\Z)
$$
is an order isomorphism (which maps $[1]_0$ to $[1]_0$). Since the index map 
$$
\delta_1:K_1(\mathcal{A}_{\theta}\rtimes_A\Z)\rightarrow K_0(\mathcal{A}_{\theta})\cong K_1(\mathcal{A}_{\theta})
$$
is a group isomorphism, we conclude from the classification of tracially AF algebras that $\mathcal{A}_{\theta}\rtimes_A\Z\cong \mathcal{A}_{\theta}$.
\end{rem}

\begin{rem}
We continue to assume $\tr(A) = 3$. Using the fact that $\mathcal{A}_{\theta}\rtimes_A\Z\cong \mathcal{A}_{\theta}$ we can construct an inductive sequence in the following way. Let $\mathcal{B}_0=\mathcal{A}_{\theta}$ and $\mathcal{B}_1 = \mathcal{A}_\theta \rtimes_A \Z$. As $\mathcal{B}_1$ is isomorphic to $\mathcal{A}_\theta$, there exists an automorphism $\alpha_1$ on $\mathcal{B}_1$ such that the crossed product $\mathcal{B}_1\rtimes_{\alpha_1} \Z =: \mathcal{B}_2$ is again isomorphic to $\mathcal{A}_\theta$. Iterating this procedure, we obtain an inductive sequence
$$
\begin{CD}
\mathcal{B}_0@>i>>\mathcal{B}_1@>i>>\mathcal{B}_2@>i>>\cdots
\end{CD},
$$
where all the connecting maps are inclusions. Let us write $\mathcal{B}$ for the limit of this sequence. Then $K_0(\mathcal{B})$ is isomorphic to $\Z+\Z\theta$ as an ordered abelian group and $K_1(\mathcal{B})$ vanishes. Since $\mathcal{B}$ is an AT algebra with real rank zero, by \cite{Ell93} (see also \cite[Theorem 3.2.6]{Ror02}) it is classified by the $K$-groups. Comparing the $K$-groups directly, we see that $\mathcal{B}$ is isomorphic to the AF algebra constructed by Effros and Shen in \cite{ES} (it was shown in \cite{PV80} that the ordered $K_0$-group of this AF algebra is isomorphic to $\Z+\Z\theta$).
\end{rem}

\begin{rem}
According to Theorem \ref{thm:main_iso}, for a fixed irrational number $\theta$ and matrices $A$ and $B\in SL_2(\Z)$ with different traces, the resulting crossed products might still be isomorphic. This is because $I_2-A^{-1}$ and $I_2-B^{-1}$ might still have the same Smith normal form. For example, take
$$
A=\begin{pmatrix}
-1& \phantom{-}1\\
 \phantom{-}0&-1
\end{pmatrix}\ {\rm and\ }
B=\begin{pmatrix}
2&1\\
 7&4
\end{pmatrix}.
$$
Then $I_2-A^{-1}$ and $I_2-B^{-1}$ have the common Smith normal form $\diag(1,4)$. In general, $I_2-A^{-1}$ and $I_2-B^{-1}$ have the common Smith normal form if and only if
\begin{enumerate}
\item The greatest common divisor of the entries of $I_2-A^{-1}$ equals to that of $I_2-B^{-1}$ up to sign, and
\item $\vert 2-\tr(A)\vert=\vert 2-\tr(B)\vert$.
\end{enumerate}

In particular, having the same trace does not guarantee the resulting crossed products to be isomorphic. For example, if we take
$$
A=\begin{pmatrix}
4&9\\
 3&7
\end{pmatrix}\ {\rm and\ }
B=\begin{pmatrix}
1&9\\
 1&10
\end{pmatrix},
$$
then $I_2-A^{-1}$ has the Smith normal form $\diag(3,3)$, while $I_2-B^{-1}$ has the Smith normal form $\diag(1,9)$.
\end{rem}

We close this section by giving a set of generators for $K_1(\mathcal{A}_{\theta}\rtimes_A\Z)$. Recall that we have the split short exact sequence
\begin{equation*}
0 \longrightarrow \bigslant{ K_1(\mathcal{A}_{\theta})}{\mathrm{im}(\id-\alpha^{-1}_{*1})}\stackrel{i_*}{\longrightarrow} K_1(\mathcal{A}_{\theta}\rtimes_A\Z) \stackrel{\delta_1}{\longrightarrow} K_0(\mathcal{A}_{\theta}) \longrightarrow 0.
\end{equation*}
We first consider the question of computing preimages of the index map $\delta_1$. The next theorem is essentially the same as \cite[Proposition 3.2.7]{Thesis}.
	
\begin{thm}\label{thm:index}
Let $\alpha \in \Aut(\mathcal{A}_{\theta})$ and $w$ be the implementing unitary in $\mathcal{A}_\theta\rtimes_\alpha \Z$. Let 
$$
\delta_1:K_1(\mathcal{A}_{\theta}\rtimes_{\alpha}\Z)\rightarrow K_0(\mathcal{A}_{\theta})
$$
be the index map in the Pimsner-Voiculescu sequence. Suppose $p$ is a projection in $\mathcal{A}_{\theta}$ and $s$ is a partial isometry in $\mathcal{A}_\theta$ such that $s^*s=p$ and $ss^*=\alpha^{-1}(p)$. Define
$$
y=s^*w^*+(1-p)\in\mathcal{A}_{\theta}\rtimes_{\alpha}\Z.
$$
Then $y$ is a unitary and $\delta_1([y]_1)=[p]_0$.
\end{thm}
\begin{proof}
We first observe that
$$
p(s^*w^*)=s^*w^*=s^*\alpha^{-1}(p)w^*=s^*w^*(w\alpha^{-1}(p)w^*)=(s^*w^*)p
$$
and that
$$
p(ws)=w(w^*pw)s=w(\alpha^{-1}(p)s)=ws=(ws)p.
$$
Then we compute
\begin{align*}
		y^*y &= (ws+(1-p))(s^*w^*+(1-p)) \\
				&= w\alpha^{-1}(p)w^*+0+0+(1-p) \\
		&= 1,		
\end{align*}
\begin{align*}
		yy^* &= (s^*w^*+(1-p))(ws+(1-p)) \\
				&= s^*s+0+0+(1-p) \\
		&= 1.		
\end{align*}
Therefore $y$ is a unitary. To give a concrete recipe of the index map $\delta_1$, let us recall the Toeplitz extension. Let $\mathcal{T}$ denote the Toeplitz algebra, which is the universal $C^*$-algebra generated by an isometry $S$. Then there exists a short exact sequence
\begin{equation*}
0 \longrightarrow  \mathcal{K}\otimes\mathcal{A}_{\theta} \stackrel {\varphi}{\longrightarrow} \mathcal{T}_{\alpha}\stackrel{\psi}{\longrightarrow} \mathcal{A}_{\theta}\rtimes_{\alpha}\Z \longrightarrow 0,
\end{equation*}
where $\mathcal{T}_{\alpha}$ is the $C^*$-subalgebra of $(\mathcal{A}_{\theta}\rtimes_{\alpha}\Z)\otimes \mathcal{T}$ generated by $(\mathcal{A}_{\theta}\rtimes_{\alpha}\Z)\otimes 1$ and $w\otimes S$. The map $\varphi$ is given by the formula
$$
\varphi(e_{i,j}\otimes a)=\alpha^i(a)V^jP(V^*)^j\;\;\;\;\;\; (a\in\mathcal{A}_{\theta}),
$$
where $V=w\otimes S$ and $P=1-VV^*=1\otimes (1-SS^*)$, and the map $\psi$ is defined by setting $\psi(b\otimes 1)=b$ for all $b\in\mathcal{A}_{\theta}\rtimes_{\alpha}\Z$ and $\psi(w\otimes S)=w$.

We apply the second part of \cite[Proposition 9.2.3]{Ror} (due to Elliott) to the six-term exact sequence arising from the Toeplitz extension. Toward this end, define
$$
x=s^*w^*\otimes S+(1-p)\otimes 1\in\mathcal{T}_{\alpha}.
$$
then 
$$
\psi(x)=\psi(s^*\otimes 1)\psi(w^*\otimes S)+\psi((1-p)\otimes 1)=y,
$$
and $x$ is a partial isometry. Indeed, we compute
\begin{align*}
		x^*x &=p\otimes SS^*+(1-p)\otimes 1
\end{align*}
and 
\begin{align*}
		xx^* &= p\otimes 1+(1-p)\otimes 1 = 1\otimes 1		.
\end{align*}
Hence $1-xx^*=0$ and 
\begin{align*}
		1-x^*x &= (p\otimes 1)(1\otimes (1-SS^*)) \\
				&= \varphi(e_{0,0}\otimes p).	
\end{align*}
Since $K_0(\mathcal{A}_{\theta})$ is identified with $K_0(\mathcal{K}\otimes \mathcal{A}_{\theta})$ via the map $[q]_0 \mapsto [e_{0,0}\otimes q]_0$, we deduce from \cite[Proposition 9.2.3]{Ror} that
$$
\delta_1([y]_1)=[p]_0,
$$
as desired.
\end{proof}

Now let $p_\theta$ be a Rieffel projection in $\mathcal{A}_\theta$ such that $[1]_0$ and $[p_\theta]_0$ generate $K_0(\mathcal{A}_\theta)$. Since $\mathcal{A}_{\theta}$ has the cancellation property  (see for example \cite{Rie83}) and $\alpha^{-1}_{*0}=\id_{K_0(\mathcal{A}_{\theta})}$, there exist a partial isometry $s_{\theta}$ such that $s^*_{\theta}s_{\theta}=p_{\theta}$ and $s_{\theta}s^*_{\theta}=\alpha^{-1}(p_{\theta})$. 
\begin{cor}\label{Ge}
If we define
$$
y_{\theta}=s^*_{\theta}w^*+(1-p_{\theta})\in\mathcal{A}_{\theta}\rtimes_A\Z,
$$
then $y_{\theta}$ is a unitary with $\delta_1([y_{\theta}]_1)=[p_{\theta}]_0$. Moreover, $\delta_1([w^*]_1)=[1]_0$.
\end{cor}
\begin{proof}
	The first assertion follows directly from the previous theorem. For the second statement, take $p= 1$ and $s=1$ in Theorem \ref{thm:index}.
\end{proof}

Now we consider the quotient $K_1(\mathcal{A}_{\theta})/\mathrm{im}(\id-\alpha^{-1}_{*1})$. According to the general theory of Smith normal norms, there exist matrices $P$ and $Q\in GL_2(\Z)$ such that
$$
P(I_2-A^{-1})Q=\diag(h_1,h_2),
$$
for some positive integers $h_1$, $h_2$ with $h_1$ dividing $h_2$ (in the case $I_2-A^{-1}$ is of rank $1$ we have $h_2=0$). We choose $[U_1]_1$ and $[U_2]_1$ to be the generators of $K_1(\mathcal{A}_{\theta})$. Then the quotient group $\bigslant{ K_1(\mathcal{A}_{\theta})}{\mathrm{im}(\diag(h_1,h_2)Q^{-1})}$ is generated by the classes represented by $[U_1']_1$ and $[U_2']_1$, where
 $$
  \begin{pmatrix}
 [U_1']_1\\
 [U_2']_1
\end{pmatrix}
=Q^{-1}
 \begin{pmatrix}
 [U_1]_1\\
 [U_2]_1
\end{pmatrix}.
 $$
Since there is a group isomorphism 
$$
\bigslant{ K_1(\mathcal{A}_{\theta})}{\mathrm{im}(P^{-1}\diag(h_1,h_2)Q^{-1})}\rightarrow \bigslant{ K_1(\mathcal{A}_{\theta})}{\mathrm{im}(\diag(h_1,h_2)Q^{-1})}
$$
given by multiplication of the matrix $P$, the quotient group  $K_1(\mathcal{A}_{\theta})/\mathrm{im}(\id-\alpha^{-1}_{*1})$ is then generated by the class represented by $[U_1'']_1$ and $[U_2'']_1$, where
\begin{equation}\label{vector}
 \begin{pmatrix}
 [U_1'']_1\\
 [U_2'']_1
\end{pmatrix}
=P^{-1}Q^{-1}
 \begin{pmatrix}
 [U_1]_1\\
 [U_2]_1
\end{pmatrix}.
\end{equation}
From these discussions we obtain the following description of a set of generators for $K_1(\mathcal{A}_{\theta}\rtimes_A\Z)$.
\begin{thm}\label{thm:K_1}
Let $\theta$ be an irrational number and $A\in SL_2(\Z)$ be a matrix of infinite order. Let $P$ and $Q$ be matrices in $GL_2(\Z)$ such that $P(I_2-A^{-1})Q=\diag(h_1,h_2)$ is the Smith normal form of $I_2-A^{-1}$ (with $h_2$ possibly being $0$). Then $K_1(\mathcal{A}_{\theta}\rtimes_A\Z)$ is generated by $[w^*]_1$, $[y_{\theta}]_1$ (given in Corollary \ref{Ge}) $[U_1'']_1$, and $[U_2'']_1$ (given by (\ref{vector}) above).
\end{thm}


\section{Morita equivalence classes of $\mathcal{A}_{\theta}\rtimes_A \Z$}

In this section we determine precisely when two crossed products of the form $A_\theta\rtimes_A\Z$ are Morita equivalent. The main tool is the following theorem obtained independently by Combes and Curto-Muhly-Williams. Roughly, the result says that if two $C^*$-algebras $\mathcal{A}$ and $\mathcal{B}$ are Morita equivalent via a bimodule $X$ and a group $G$ acts on both $\mathcal{A}$ and $\mathcal{B}$, then the crossed products $\mathcal{A}\rtimes G$ and $\mathcal{B}\rtimes G$ are Morita equivalent provided we can find a $G$-action on $X$ which is compatible with the actions on $\mathcal{A}$ and $\mathcal{B}$. Also see \cite{EKQR06} for a more categorical approach.

\begin{thm} \cite[p. 299]{Com84},\cite[Theorem 1]{CMW84}
	Let $\mathcal{A},\mathcal{B}$ be $C^*$-algebras, $G$ a locally compact group, and $\alpha:G\to \Aut(\mathcal{A})$ and $\beta:G\to \Aut(\mathcal{B})$ be continuous group actions. 
	Suppose there is a $\mathcal{B}$-$\mathcal{A}$-imprimitivity bimodule $X$ and a strongly continuous action of $G$ on $X$, $\{ \tau_g \}_{g\in G}$, such that for all $x,y\in X$ and $g\in G$ we have
	\begin{enumerate}
		\item $\ip{ \tau_g(x), \tau_g(y) }_\mathcal{A} = \alpha_g( \ip{x,y}_\mathcal{A} )$, and
		\item $_\mathcal{B} \ip{ \tau_g(x), \tau_g(y) } = \beta_g( _\mathcal{B}\ip{x,y} )$.
	\end{enumerate}
	Then the crossed products $\mathcal{A}\rtimes_{\alpha} G$ and $\mathcal{B}\rtimes_{\beta} G$ are Morita equivalent.
\end{thm}

A standard completion argument shows that in the theorem it is enough to have an action of $G$ on some pre-imprimitivity bimodule linking dense $*$-subalgebras of $\mathcal{A}$ and $\mathcal{B}$. Here is the precise statement.

\begin{prop} \label{prop:Combes}
	Let $\mathcal{A},\mathcal{B}$ be $C^*$-algebras, $G$ a locally compact group, and $\alpha:G\to \Aut(\mathcal{A})$ and $\beta:G\to \Aut(\mathcal{B})$ be continuous group actions. Suppose there exist dense $*$-subalgebras $\mathcal{A}_0\subseteq \mathcal{A}$ and $\mathcal{B}_0\subseteq \mathcal{B}$, a $\mathcal{B}_0$-$\mathcal{A}_0$-pre-imprimitivity bimodule $X_0$, and a strongly continuous action of $G$ on $X$, $\{\tau_g\}_{g\in G}$ such that for all $x,y\in X_0$ and $g\in G$ we have
	\begin{enumerate}
		\item $\ip{ \tau_g(x), \tau_g(y) }_\mathcal{A} = \alpha_g( \ip{x,y}_\mathcal{A} )$, and
		\item $_\mathcal{B} \ip{ \tau_g(x), \tau_g(y) } = \beta_g( _\mathcal{B}\ip{x,y} )$.
	\end{enumerate}
	Then the crossed products $\mathcal{A}\rtimes_{\alpha} G$ and $\mathcal{B}\rtimes_{\beta} G$ are Morita equivalent.
\end{prop}

Let $\theta\in \R\setminus \Q$ and let $\alpha:SL_2(\Z)\curvearrowright \mathcal{A}_{\theta}$ and $\beta:SL_2(\Z)\curvearrowright \mathcal{A}_{\frac{1}{\theta}}$ be the actions defined in Section 2. We wish to apply this result to the case that $\mathcal{A} = \mathcal{A}_{\theta}$, $\mathcal{A}_0 = \mathcal{S}(\Z^2, \omega_\theta)$, $\mathcal{B}= \mathcal{A}_{ \frac{1}{\theta} }$, $\mathcal{B}_0 = \mathcal{S}(\Z^2, \omega_{ \frac{1}{\theta }})$, and $X_0 = \mathcal{S}(\R)$. We write
$$
J = \begin{pmatrix}
0 & 1 \\
-1 & 0
\end{pmatrix} \;\;\;\;\;\;\text{and}\;\;\;\;\;\; P = \begin{pmatrix}
1 & 0 \\
1 & 1
\end{pmatrix}.
$$
It is well-known that $J$ and $P$ generate the group $SL_2(\Z)$. Our first goal is to show that there exist unitary operators $S_J$ and $S_P$ on $L^2(\R)$ such that for all $f,g\in \mathcal{S}(\R)$ we have
$$
\ip{ S_J(f), S_J(g) }_{\mathcal{A}^\infty} = \alpha_J ( \ip{f,g}_{\mathcal{A}^\infty}  ),\;\;\;\;\;\; _{\mathcal{B}^\infty}\ip{ S_J(f), S_J(g) } = \beta_{J^{-1}}( _{\mathcal{B}^\infty}\ip{f,g} )
$$
and 
$$
\ip{ S_P(f), S_P(g) }_{\mathcal{A}^\infty} = \alpha_P ( \ip{f,g}_{\mathcal{A}^\infty}  ),\;\;\;\;\;\; _{\mathcal{B}^\infty}\ip{ S_P(f), S_P(g) } = \beta_{P^{-1}}( _{\mathcal{B}^\infty}\ip{f,g} )
$$

\begin{defn} \label{defn:S_J}
	Define $S_J:\mathcal{S}(\R)\to \mathcal{S}(\R)$ by
	$$
	S_J(f)(x) := \theta^{-1/2}\int_{\R} e(xy/\theta)f(y)\;dy
	$$
	and $S_P:\mathcal{S}(\R)\to \mathcal{S}(\R)$ by
	$$
	S_P(f)(x) := e(-x^2/(2\theta))f(x).
	$$
\end{defn}

The operator $S_J$ is precisely the operator $W$ introduced in \cite[p. 647]{Wal00}, where it is shown that $S_J$ has period four and its square is equal to the flip operator, i.e., the operator sending $f(x)$ to $f(-x)$ (see also \cite[Proposition 4.6]{EGLN15}). Indeed, we have
$$
S_J(f)(x) = \theta^{-1/2}\hat{f}(-x/\theta),
$$
where $\hat{f}$ is the Fourier transform of $f$. Since the Fourier transform extends to a unitary operator on $L^2(\R)$, so does the operator $S_J$.

The operator $S_P$ is a so-called \emph{metaplectic operator} associated to the matrix $P$ (see \cite[Chapter 7]{dGos11}). It is clear from the definition that $S_P$ extends to a unitary operator on $L^2(\R)$. 

We use the same notations $S_J$ and $S_P$ for the extensions to $L^2(\R)$.  In what follows, we keep the notations in Section 2.4., and we will use Proposition \ref{prop:action} and Proposition \ref{prop:module_action} freely.

The next two lemmas show that the operators $S_J$ and $S_P$ are compatible with the automorphisms $\alpha_J$ and $\alpha_P$, respectively.

\begin{lem} \label{lem:action-comp}
	For all $f\in \mathcal{S}(\R)$ and $\ell\in \Z^2$, we have 
	$$
	S_J(f.U_\ell) = S_J(f). \alpha_J(U_\ell),\;\;\;\;\;\; S_J(V_\ell.f) = \beta_{J^{-1}}(V_\ell). S_J(f).
	$$
\end{lem}
\begin{proof}
	It suffices to verify the equalities on generators $U_1, U_2$ and $V_1, V_2$. The result follows from a series of direct computations.
	\begin{align*}
		S_J(f.U_1)(x) &= \theta^{-1/2}\int_\R e(xy/\theta) (f.U_1)(y) \;dy = \theta^{-1/2}\int_\R e(xy/\theta)f(y+\theta)\;dy \\
		&= \theta^{-1/2}\int_{\R} e(x(y-\theta)/\theta)f(y)\;dy = \theta^{-1/2}e(-x)\int_\R e(xy/\theta) f(y) \;dy \\
		&= e(-x)S_J(f)(x) = [S_J(f).U_2^{-1}](x) \\
		&= [S_J(f).\alpha_J(U_1)](x).
	\end{align*}
	\begin{align*}
		S_J(f.U_2)(x) &= \theta^{-1/2}\int_{\R} e(xy/\theta)(f.U_2)(y)\;dy = \theta^{-1/2}\int_\R e(xy/\theta)e(y)f(y)\;dy \\
		&= \theta^{-1/2}\int_\R e( (x+\theta)y/\theta )f(y)\;dy = S_J(f)(x+\theta) \\
		&= [S_J(f).U_1](x) \\
		&= [S_J(f).\alpha_J(U_2)](x).		
	\end{align*}
	\begin{align*}
		S_J(V_1.f)(x) &= \theta^{-1/2}\int_\R e(xy/\theta)(V_1.f)(y)\;dy = \theta^{-1/2}\int_{\R} e(xy/\theta) f(y+1)\;dy \\
		&= \theta^{-1/2}\int_{\R}e(x(y-1)/\theta)f(y)\;dy = \theta^{-1/2}e(-x/\theta) \int_{\R} e(xy/\theta)f(y)\;dy \\
		&= e(-x/\theta) S_J(f)(x) = [V_2.S_J(f)](x) \\
		&= [\beta_{J^{-1}}(V_1).S_J(f)](x).
	\end{align*}
	\begin{align*}
		S_J(V_2.f)(x) &= \theta^{-1/2}\int_\R e(xy/\theta)(V_2.f)(y)\;dy = \theta^{-1/2}\int_\R e(xy/\theta)e(-y/\theta)f(y)\;dy \\
		&= \theta^{-1/2}\int_\R e((x-1)y/\theta)f(y)\;dy = S_J(f)(x-1) \\
		&= [V_1^{-1}.S_J(f)](x)  \\
		&= [ \beta_{J^{-1}}(V_2). S_J(f) ](x).
	\end{align*}
\end{proof}

\begin{lem} 
	For all $f\in \mathcal{S}(\R)$ and $\ell\in \Z^2$, we have
	$$
	S_P(f.U_\ell) = S_P(f). \alpha_P(U_\ell),\;\;\;\;\;\; S_P(V_\ell.f) = \beta_{P^{-1}}(V_\ell). S_P(f).
	$$
\end{lem}
\begin{proof}
	Again we check the claims for generators $U_1,U_2$ and $V_1,V_2$.
	\begin{align*}
		[S_P(f).\alpha_P(U_1)](x) &= \left[ S_P(f).U_{ \tiny \begin{pmatrix} 
				1 \\ 1
			\end{pmatrix} }  \right](x) = e(\theta/2)e(x)S_P(f)(x+\theta) \\
			&= e(\theta/2)e(x)e(-(x+\theta)^2/(2\theta))f(x+\theta) \\
			&= e(-x^2/(2\theta))f(x+\theta) \\
			&= S_P(f.U_1)(x).
		\end{align*}
		\begin{align*}
			[S_P(f).\alpha_P(U_2)](x) &= [S_P(f).U_2](x) = e(x)S_P(f)(x) \\
			&= e(x)  e(-x^2/(2\theta))f(x) \\
			&= e(-x^2/(2\theta))(f.U_2)(x) \\
			&= S_P(f.U_2)(x).
	\end{align*}
	\begin{align*}
			[\beta_{P^{-1}}(V_1).S_P(f)](x) &= \left[ V_{ \tiny   \begin{pmatrix}
					1 \\ -1
				\end{pmatrix} }. S_P(f) \right](x) = e(1/(2\theta))e(x/\theta)S_P(f)(x+1) \\
				&= e(1/(2\theta))e(x/\theta) e(-(x+1)^2/(2\theta)) f(x+1) \\
				&= e(-x^2/(2\theta))f(x+1) \\
				&= S_P(V_1.f)(x).
	\end{align*}
	\begin{align*}
				\left[  \beta_{p^{-1}}(V_2).S_P(f)  \right](x) &= [V_2.S_P(f)](x) = e(-x/\theta) S_P(f)(x) \\
				&= e(-x/\theta)e(-x^2/(2\theta))f(x) \\
				&= e(-x^2/(2\theta))(V_2.f)(x) \\
				&= S_P(V_2.f)(x).
	\end{align*}
\end{proof}

Next we show the compatibility among inner products, which is what we need in order to apply Proposition \ref{prop:Combes}.

\begin{prop} \label{prop:ip-comp}
	For all $f,g\in \mathcal{S}(\R)$ we have 
	$$
	\ip{ S_J(f), S_J(g) }_{\mathcal{A}^\infty} = \alpha_J ( \ip{f,g}_{\mathcal{A}^\infty}  ),\;\;\;\;\;\; _{\mathcal{B}^\infty}\ip{ S_J(f), S_J(g) } = \beta_{J^{-1}}( _{\mathcal{B}^\infty}\ip{f,g} )
	$$
	and 
	$$
	\ip{ S_P(f), S_P(g) }_{\mathcal{A}^\infty} = \alpha_P ( \ip{f,g}_{\mathcal{A}^\infty}  ),\;\;\;\;\;\; _{\mathcal{B}^\infty}\ip{ S_P(f), S_P(g) } = \beta_{P^{-1}}( _{\mathcal{B}^\infty}\ip{f,g} )
	$$	
\end{prop}
\begin{proof}
	We only check the identities for $J$, as for $P$ it is completely analogous. Replacing $f$ by $S_J^{-1}(f)$, it suffices to show that
	$$
	\ip{ f, S_J(g) }_{\mathcal{A}^\infty} = \alpha_J ( \ip{ S_J^{-1}(f),g}_{\mathcal{A}^\infty}  ),\;\;\;\;\;\; _{\mathcal{B}^\infty}\ip{ f, S_J(g) } = \beta_{J^{-1}}( _{\mathcal{B}^\infty}\ip{ S_J^{-1}(f),g} ).
	$$
	For convenience let us write $\lambda_A := \theta e(mn\theta/2)$ and $\lambda_B := e(mn/(2\theta))$. Using Proposition \ref{prop:L2-trick} and Lemma \ref{lem:action-comp}, for each $\ell =  \begin{pmatrix}
	m \\ n
	\end{pmatrix} \in \Z^2$ we have 
	\begin{align*}
		\ip{ f, S_J(g) }_{\mathcal{A}^\infty}(\ell) &= \lambda_A \ip{ S_J(g).U_{-\ell}, f}_{L^2(\R)} \\
		&= \lambda_A \ip{ S_J(g). \alpha_J (\alpha_{J^{-1}}(U_{-\ell})  ), f }_{L^2(\R)} \\
		&= \lambda_A\ip{  S_J( g.\alpha_{J^{-1}}(U_{-\ell}) ), f }_{L^2(\R)} \\
		&= \lambda_A\ip{ S_J(g. U_{-J^{-1}\ell}) , f }_{L^2(\R)} \\
		&= \lambda_A\ip{ g.U_{-J^{-1}\ell}, S_J^{-1}(f) }_{L^2(\R)} \\
		&= \ip{S_J^{-1}(f), g}_{\mathcal{A}^\infty}(J^{-1}\ell) \\
		&= \alpha_J( \ip{S_J^{-1}(f), g  }  )(\ell).
	\end{align*}
	Similarly, we have
	\begin{align*}
		_{\mathcal{B}^\infty} \ip{ f, S_J(g) }(\ell) &= \lambda_B \ip{f, V_\ell.S_J(g) }_{L^2(\R)} \\
		&= \lambda_B \ip{f, \beta_{J^{-1}}(\beta_J(V_\ell)). S_J(g) }_{L^2(\R)} \\
		&= \lambda_B\ip{ f, S_J(\beta_J(V_\ell).g) }_{L^2(\R)} \\
		&= \lambda_B \ip{  f, S_J(V_{J\ell}.g) }_{L^2(\R)} \\
		&= \lambda_B \ip{ S_J^{-1}(f), V_{J\ell}.g	}_{L^2(\R)} \\
		&= {}_{\mathcal{B}^\infty} \ip{ S_J^{-1}(f), g}(J\ell) \\
		&= \beta_{J^{-1}}\left(  _{\mathcal{B}^\infty}\ip{  S_J^{-1}(f), g}   \right)	(\ell)
	\end{align*}
	This completes the proof.
\end{proof}

We wish to use the fact that $J$ and $P$ generate the group $SL_2(\Z)$ to construct $\Z$-actions on $\mathcal{S}(\R)$. 
We define $S_{J^{-1}}$ and $S_{P^{-1}}$ to be the inverses of $S_J$ and $S_P$ in $\mathcal{U}(L^2(\R))$, respectively.

\begin{defn}
	Let $A \in SL_2(\Z)$. Then $A$ can be written as $A = W_1W_2\cdots W_n$, where each $W_k$ belongs to $\{J,P,J^{-1}, P^{-1}\}$. Define the operator $S_A:\mathcal{S}(\R)\to \mathcal{S}(\R)$ by
	$$
	S_A := S_{W_1}\circ S_{W_2}\circ\cdots \circ S_{W_n}.
	$$
\end{defn}

\begin{rem}
	Strictly speaking, the operator $S_A$ depends on how we write $A$ as a product of generators. However this would not matter for our consideration, hence we omit this dependence in the notation. In fact, the operator $S_A$ is uniquely determined up to sign. See \cite[Chapter 7]{dGos11} for a comprehensive discussion.
\end{rem}

Now we are ready to prove the main theorem of this section. Let us write $T := \begin{pmatrix}
-1 & 0 \\
0 & 1
\end{pmatrix}$. Note that $T^{-1} = T$ (hence $T^2 = I$).

\begin{thm} \label{thm:morita_integer_reciprocal}
	Let $\theta \in \R\setminus \Q$ and $A\in SL_2(\Z)$. Then the crossed product $\mathcal{A}_{\theta} \rtimes_A \Z$ is Morita equivalent to the crossed product $\mathcal{A}_{ \frac{1}{\theta}  } \rtimes_B \Z$, where $B := TAT$.
\end{thm}
\begin{proof}
	Let $\tau:\Z\to \ip{ S_A} \subseteq \mathcal{U}(L^2(\R))$ be the group homomorphism sending $n$ to $(S_A)^n$. By definition and Proposition \ref{prop:ip-comp}, for all $f,g\in \mathcal{S}(\R)$ we have
	\begin{align*}
		\ip{ S_A(f), S_A(g) }_{\mathcal{A}^\infty} &= \ip{ S_{W_1}\circ \cdots S_{W_n}(f), S_{W_1}\circ \cdots S_{W_n}(g) }_{\mathcal{A}^\infty} \\
		&= \alpha_{W_1} ( \ip{ S_{W_2}\circ \cdots S_{W_n}(f), S_{W_2}\circ \cdots S_{W_n}(f) }_{\mathcal{A}^\infty} ) \\
		&= \cdots \\
		&= \alpha_{W_1}\cdots \alpha_{W_n}( \ip{f,g}_{\mathcal{A}^\infty}  ) \\
		&= \alpha_{W_1\cdots W_n}( \ip{f,g}_{\mathcal{A}^\infty} ) \\
		&= \alpha_A( \ip{f,g}_{\mathcal{A}^\infty} ).
	\end{align*}
	Replacing $f$ by $S_A^{-1}(f)$ and $g$ by $S_A^{-1}(g)$, the identity above becomes
	$$
	\ip{f,g}_{\mathcal{A}^\infty} = \alpha_A( \ip{S_A^{-1}(f), S_A^{-1}(g) }_{\mathcal{A}^\infty} ).
	$$
	Applying $\alpha_{A^{-1}}$ to both sides we get
	$$
	\alpha_{A^{-1}}( \ip{f,g}_{\mathcal{A}^\infty}  ) = \ip{ S_A^{-1}(f), S_A^{-1}(g) }_{\mathcal{A}^\infty}.
	$$
	Therefore for any $n\in \Z$, we have
	\begin{align*}
		\ip{ \tau_n(f), \tau_n(g) }_{\mathcal{A}^\infty} &= \ip{ (S_A)^n(f), (S_A)^n(g) }_{\mathcal{A}^\infty} = (\alpha_A)^n( \ip{f,g}_{\mathcal{A}^\infty} ) \\
		&= \alpha_{A^n}( \ip{f,g}_{\mathcal{A}^\infty} ).
	\end{align*}
	Similarly we have
	\begin{align*}
		_{\mathcal{B}^\infty}\ip{ S_A(f), S_A(g) } &= {}_{\mathcal{B}^\infty} \ip{ S_{W_1}\circ \cdots S_{W_n}(f), S_{W_1}\circ \cdots S_{W_n}(g) } \\
		&= \beta_{ W_1^{-1} } ( {}_{\mathcal{B}^\infty} \ip{ S_{W_2}\circ \cdots S_{W_n}(f), S_{W_2}\circ \cdots S_{W_n}(g) } ) \\
		&= \cdots \\
		&= \beta_{W_1^{-1}}\cdots \beta_{W_n^{-1}}(  {}_{\mathcal{B}^\infty}\ip{f,g}  ) \\
		&= \beta_{W_1^{-1}\cdots W_n^{-1} }( {}_{\mathcal{B}^\infty} \ip{f,g} ).
	\end{align*}
	Note that for matrices of the form $W = \begin{pmatrix}
	a & b \\
	c & a
	\end{pmatrix}$, we have
	$$
	W^{-1} = \begin{pmatrix}
	a & -b \\
	-c & a 
	\end{pmatrix} = TWT.
	$$
	Since both generators $J$ and $P$ (and their inverses) have this form, we have 
	$$
	W_1^{-1}W_2^{-1}\cdots W_n^{-1} = (TW_1T)(TW_2T)\cdots (TW_nT) = T(W_1W_2\cdots W_n)T = TAT = B,
	$$
	and hence
	$$
	_{\mathcal{B}^\infty}\ip{ S_A(f), S_A(g) } = \beta_B( {}_{\mathcal{B}^\infty}\ip{f,g} ) .
	$$
	As before, replacing $f$ and $g$ by $S_A^{-1}(f)$ and $S_A^{-1}(g)$, and applying $\beta_{B^{-1}}$ to both sides we obtain
	$$
	_{\mathcal{B}^\infty}\ip{ S_A^{-1}(f), S_A^{-1}(g) } = \beta_{B^{-1}}( {}_{\mathcal{B}^\infty}\ip{f,g}  ).
	$$
	Therefore for each $n\in \Z$ we have
	\begin{align*}
		{}_{\mathcal{B}^\infty}\ip{ \tau_n(f), \tau_n(g) } &= {}_{\mathcal{B}^\infty}\ip{ (S_A)^n(f), (S_A)^n(g) } = (\beta_B)^n( \ip{f,g}_{\mathcal{A}^\infty} ) \\
		&= \beta_{B^n}( \ip{f,g}_{\mathcal{A}^\infty} ).
	\end{align*}
	Now we are done because the action $\tau:\Z\curvearrowright \mathcal{S}(\R)$ satisfies all the assumptions of Proposition \ref{prop:Combes}.
\end{proof}

Using the theorem above, we can now give a precise description of the Morita equivalence classes. For two irrational numbers $\lambda$ and $\mu$, we write $\lambda\sim_{Mob} \mu$ if they belong to the same $GL_2({\Z})$-orbit. More precisely, $\lambda\sim_{Mob} \mu$ if and only if $\mu = \frac{a\lambda + b}{c\lambda + d}$ for some matrix $\begin{pmatrix}
a & b \\
c & d
\end{pmatrix}$ in $GL_2(\Z)$.

\begin{thm}  \label{thm:main_morita}
	Let $\theta, \theta'$ be irrational numbers and $A,B\in SL_2(\Z)$ be matrices of infinite order. Then the following are equivalent:
	\begin{enumerate}
		\item $\mathcal{A}_{\theta}\rtimes_A \Z$ and $\mathcal{A}_{\theta'} \rtimes_B \Z$ are Morita equivalent;
		\item $\theta \sim_{Mob} \theta'$ and $I-A^{-1}\sim_{eq}I-B^{-1}$.
	\end{enumerate}
\end{thm}
\begin{proof}
	$(2)\Longrightarrow (1)$: It suffices to consider the cases $\theta' = \frac{1}{\theta}$ and $\theta' = \theta + 1$. Note that by Theorem \ref{thm:main_iso} in the latter case we actually get an isomorphism. Therefore assume $\theta' = \frac{1}{\theta}$. Theorem $\ref{thm:morita_integer_reciprocal}$ says that $\mathcal{A}_{\theta} \rtimes_A \Z$ is Morita equivalent to $\mathcal{A}_{\frac{1}{\theta}} \rtimes_{(TAT)} \Z$. Note that
	$$
	I-(TAT)^{-1} = I - TA^{-1}T = T(I-A^{-1})T.
	$$
	Therefore $I-(TAT)^{-1} \sim_{eq} I-A^{-1}$ and by Theorem \ref{thm:main_iso} we have 
	$$
	\mathcal{A}_{\frac{1}{\theta}}\rtimes_{TAT} \Z \cong \mathcal{A}_{\frac{1}{\theta}}\rtimes_{A} \Z \cong \mathcal{A}_{\frac{1}{\theta}}\rtimes_{B} \Z.
	$$ 
	In conclusion, we have
	$$
	\mathcal{A}_\theta \rtimes_A \Z \sim_{\text{Morita}} \mathcal{A}_{\frac{1}{\theta}}\rtimes_{TAT} \Z \cong \mathcal{A}_{\frac{1}{\theta}}\rtimes_{B} \Z.
	$$
	
	$(1)\Longrightarrow (2)$: Let us write $\mathcal{C} := \mathcal{A}_{\theta} \rtimes_A \Z$ and $\mathcal{D}:= \mathcal{A}_{\theta'}\rtimes_B \Z$. As in Theorem \ref{thm:main_iso} the condition on matrix equivalence comes from the $K_1$ groups. Since $\mathcal{C}$ and $\mathcal{D}$ are monotracial and the images under the tracial states are $\Z+\theta\Z$ and $\Z+\theta'\Z$, respectively, we get $\theta\sim_{Mob} \theta'$ in exactly the same way as in the proof of \cite[Theorem 4]{Rie81}. 
	
	For the reader's convenience we briefly recall the main ingredients in the proof of \cite[Theorem 4]{Rie81}. Let $\tau$ be the unique trace on $\mathcal{C}$, and let $X$ be a $\mathcal{C}$-$\mathcal{D}$-imprimitivity bimodule.  Define a positive tracial functional $\tau_X$ on $\mathcal{D}$ by
	$$
	\tau_X( \ip{x,y}_\mathcal{D}  ) := \tau( _\mathcal{C}\ip{y,x} )\;\;\;\;\;\; (x,y\in X).
	$$
	By \cite[Corollary 2.6]{Rie81}, we have $\tau_*( K_0(\mathcal{C})  ) = (\tau_X)_*( K_0(\mathcal{D}) )$. Since $\mathcal{D}$ has a unique trace, $\tau_X$ must be a scale multiple of that trace. This implies that
	$$
	\Z + \theta \Z = r(\Z + \theta' \Z )
	$$
	for some positive real number $r$. Now an elementary calculation shows that $\theta$ and $\theta'$ are in the same $GL_2(\Z)$ orbit (see \cite[p. 425]{Rie81}).
\end{proof}

\begin{rem} \label{rem:alt_pf}
	Here we give an alternative proof of $(1)\Longrightarrow (2)$ in Theorem \ref{thm:main_morita}. For notational convenience we again write $\mathcal{C} := \mathcal{A}_{\theta} \rtimes_A \Z$ and $\mathcal{D}:= \mathcal{A}_{\theta'}\rtimes_B \Z$. Suppose $\mathcal{C}$ and $\mathcal{D}$ are Morita equivalent. Then there is an order isomorphism $f:(K_0(\mathcal{C}), K_0(\mathcal{C})^+)\to (K_0(\mathcal{D}), K_0(\mathcal{D})^+)$. Let $\tau_\mathcal{C}$ and $\tau_\mathcal{D}$ be the unique traces on $\mathcal{C}$ and $\mathcal{D}$, respectively. By Theorem 5.2.2 in \cite{Ror}, the ordered group $(K_0(\mathcal{C}), K_0(\mathcal{C})^+, [1_\mathcal{C}]_0))$ has a unique state, i.e., the induced map of $\tau_\mathcal{C}$. This implies that $(\tau_\mathcal{D})_*\circ f$ is a scalar multiple of $(\tau_\mathcal{C})_*$, and we arrive at the identity
	$$
	\Z +\theta \Z = r(\Z+\theta' \Z)
	$$ 
	for some positive real number $r$.
\end{rem}


\section{Morita equivalence classes of $\mathcal{A}_{\theta}\rtimes_\alpha F$}

In this final section we determine the Morita equivalence classes for crossed prdocuts of the form $\mathcal{A}_\theta \rtimes_\alpha F$, where $F$ is a finite subgroup of $SL_2(\Z)$ (recall that the isomorphism classes were studied in \cite{ELPW10}). The argument is similar to the one we used in the previous section.

Let $F$ be a finite subgroup of $SL_2(\Z)$. Up to conjugacy $F$ is isomorphic to one of $\Z_2$, $\Z_3$, $\Z_4$, or $\Z_6$. In any case, let $\alpha:F\to \Aut(\mathcal{A}_{\theta})$ be the restriction of the action of $SL_2(\Z)$ on $\mathcal{A}_{\theta}$. Let us first recall the main result of \cite{ELPW10}.

\begin{thm} \cite[Theorem 0.1]{ELPW10} \label{thm:ELPW}
	Let $\theta, \theta'$ be irrational numbers and $F$ be a finite subgroup of $SL_2(\Z)$. Then $\mathcal{A}_{\theta}\rtimes_\alpha F$ is a simple monotracial AF algebra. For $F=\Z_k$ ($k=2,3,4,6$), the $K_0$ groups are given by
	\begin{align*}
	K_0(\mathcal{A}_{\theta} \rtimes_\alpha \Z_2) &\cong \Z^6,\;\;\;\;\;\; K_0(\mathcal{A}_{\theta} \rtimes_\alpha \Z_3) \cong \Z^8 \\
	K_0(\mathcal{A}_{\theta} \rtimes_\alpha \Z_4) &\cong \Z^9,\;\;\;\;\;\; K_0(\mathcal{A}_{\theta} \rtimes_\alpha \Z_6) \cong \Z^{10}.
	\end{align*}
	In any case, the image of $K_0(\mathcal{A}_{\theta}\rtimes_\alpha F)$ under the unique tracial state is equal to $\frac{1}{k}(\Z+\theta\Z)$. As a consequence, the following are equivalent:
	\begin{enumerate}
		\item $\mathcal{A}_{\theta}\rtimes_\alpha \Z_k$ and $\mathcal{A}_{\theta'} \rtimes_\alpha \Z_{k'}$ are $*$-isomorphic;
		\item $\theta = \pm \theta' \pmod \Z$ and $k = k'$.
	\end{enumerate}
\end{thm}

As before, to show the Morita equivalence of two crossed products we aim to construct suitable action of $F$ on $\mathcal{S}(\R)$. In fact the only case that remains to consider is when $F = \Z_6$. Let $H = \begin{pmatrix}
1 & 1 \\
-1 & 0
\end{pmatrix}$ and $H' = \begin{pmatrix}
1 & -1 \\
1 & 0
\end{pmatrix}$. Observe that both $H$ and $H'$ have order six. Define the operator $S_H:\mathcal{S}(\R)\to \mathcal{S}(\R)$ by
$$
S_H(f)(x) := e^{\pi i/12}\theta^{-1/2}\int_\R e( (2xy-y^2)/(2\theta)  ) f(y)\;dy.
$$
Note that $S_H$ is the hexic transform defined by Walters in \cite{Wal04} (see also \cite[Proposition 4.6]{ELPW10}). It was shown in \cite{Wal04} that $S_H$ extends to a unitary operator on $L^2(\R)$ that has period six. 

\begin{prop}
	For all $f\in \mathcal{S}(\R)$ and $\ell\in \Z^2$, we have
	$$
	S_H(f.U_\ell) = S_H(f).\alpha_H(U_\ell),\;\;\;\;\;\;  S_H(V_\ell.f)(x) = \beta_{H'}(V_\ell).S_H(f).
	$$
	As a consequence,
	$$
	\ip{ S_H(f), S_H(g) }_{\mathcal{A}^\infty} = \alpha_H ( \ip{f,g}_{\mathcal{A}^\infty}  ),\;\;\;\;\;\; _{\mathcal{B}^\infty}\ip{ S_H(f), S_H(g) } = \beta_{H'}( _{\mathcal{B}^\infty}\ip{f,g} ).
	$$
\end{prop}
\begin{proof}
	We check the first two identities on the generators $U_1$, $U_2$ of $\mathcal{A}_\theta$ and the generators $V_1$, $V_2$ of $\mathcal{A}_{ \frac{1}{\theta}}$. For $U_1$ we have
	\begin{align*}
	S_H(f.U_1)(x) &= e^{\pi i/12}\theta^{-1/2}\int_{\R} e( (2xy-y^2)/(2\theta)  ) (f.U_1)(y)\;dy \\
	&= e^{\pi i/12}\theta^{-1/2}\int_{\R} e( (2xy-y^2)/(2\theta)  ) f(y+\theta)\;dy \\
	&= e^{\pi i/12}\theta^{-1/2}\int_{\R} e( (2x(y-\theta)-(y-\theta)^2)/(2\theta)  ) f(y)\;dy \\
	&= e^{\pi i/12}\theta^{-1/2}e(-\theta/2)e(-x)\int_{\R} e( (2xy-y^2)/(2\theta)  ) e(y)f(y)\;dy \\
	&= e^{\pi i/12}\theta^{-1/2}e(-\theta/2)e(-x)\int_{\R} e( (2(x+\theta)y-y^2)/(2\theta)  ) f(y)\;dy \\
	&= e^{\pi i/12}\theta^{-1/2}e(-\theta/2)e(-x) S_H(f)(x+\theta) \\
	&= \left[ S_H(f).U_{ \tiny \begin{pmatrix}
			1 \\ -1
		\end{pmatrix}   }   \right ](x) \\
	&= [S_H(f).\alpha_H(U_1)](x).
	\end{align*}
	For the generator $U_2$, we compute
	\begin{align*}
		S_H(f.U_2)(x) &= e^{\pi i/12}\theta^{-1/2}\int_\R e( (2xy-y^2)/(2\theta)  ) (f.U_2)(y)\;dy \\
		&=  e^{\pi i/12}\theta^{-1/2}\int_\R e( (2xy-y^2)/(2\theta)  ) e(y)f(y)\;dy \\
		&=  e^{\pi i/12}\theta^{-1/2}\int_\R e( (2(x+\theta)y-y^2)/(2\theta)  ) f(y)\;dy \\
		&= S_H(f)(x+\theta) = [S_H(f).U_1](x) \\
		&= [S_H(f).\alpha_H(U_2)](x).
	\end{align*}
	For the generator $V_1$ (of $\mathcal{A}_{\frac{1}{\theta}}$), 
	\begin{align*}
		S_H(V_1.f)(x) &=  e^{\pi i/12}\theta^{-1/2}\int_\R e( (2xy-y^2)/(2\theta)  ) (V_1.f)(y)\;dy \\
		&= e^{\pi i/12}\theta^{-1/2}\int_\R e( (2xy-y^2)/(2\theta)  ) f(y+1)\;dy \\
		&= e^{\pi i/12}\theta^{-1/2}\int_\R e( (2x(y-1)-(y-1^2)/(2\theta)  ) f(y)\;dy \\
		&= e^{\pi i/12}\theta^{-1/2} e(-x/\theta) e(-1/(2\theta))  \int_\R e( (2xy-y^2)/(2\theta)  ) e(y/\theta) f(y)\;dy \\
		&= e^{\pi i/12}\theta^{-1/2} e(-x/\theta) e(-1/(2\theta))  \int_\R e( (2(x+1)y-y^2)/(2\theta)  ) f(y)\;dy \\
		&= e(-x/\theta) e(-1/(2\theta)) S_H(f)(x+1) \\
		&= \left[ V_{ \tiny \begin{pmatrix}
				1 \\ 1
			\end{pmatrix}  }. S_H(f) \right] (x) \\
		&= [ \beta_{H'}(V_1).S_H(f) ](x).
	\end{align*}
	Finally, for $V_2$ we have
	\begin{align*}
		S_H(V_2.f)(x) &=  e^{\pi i/12}\theta^{-1/2}\int_\R e( (2xy-y^2)/(2\theta)  ) (V_2.f)(y)\;dy \\ 
		&=  e^{\pi i/12}\theta^{-1/2}\int_\R e( (2xy-y^2)/(2\theta)  ) e(-y/\theta)f(y)\;dy \\
		&=  e^{\pi i/12}\theta^{-1/2}\int_\R e( (2(x-1)y-y^2)/(2\theta)  ) f(y)\;dy \\
		&= S_H(f)(x-1) = \left[  V_{ \tiny \begin{pmatrix}
				-1 \\ 0
			\end{pmatrix} }. S_H(f)   \right](x) \\
		&= [ \beta_{H'}(V_2).S_H(f) ](x).
	\end{align*}
	The second statement follows from the first as in the proof of Proposition \ref{prop:ip-comp}.
\end{proof}

\begin{thm} \label{thm:mortia_finite}
	Let $\theta, \theta'$ be irrational numbers and $k, k'\in \{2,3,4,6\}$. Then the following are equivalent:
	\begin{enumerate}
		\item $\mathcal{A}_{\theta}\rtimes_\alpha \Z_k$ and $\mathcal{A}_{\theta'} \rtimes_\alpha \Z_{k'}$ are Morita equivalent;
		\item $\theta \sim_{Mob} \theta'$ and $k = k'$.
	\end{enumerate}
\end{thm}
\begin{proof}
	$(2)\Longrightarrow (1)$: As in the proof of Theorem \ref{thm:main_morita} it suffices to consider the case $\theta' = \frac{1}{\theta}$. In view of Proposition \ref{prop:Combes}, we only need to look at the cases $k=4$ and $k=6$. Recall that the operator $S_J$ in the previous section (see Definition \ref{defn:S_J}) has order four and satisfies
	$$
	\ip{ S_J(f), S_J(g) }_{\mathcal{A}^\infty} = \alpha_J ( \ip{f,g}_{\mathcal{A}^\infty}  ),\;\;\;\;\;\; _{\mathcal{B}^\infty}\ip{ S_J(f), S_J(g) } = \beta_{J^{-1}}( _{\mathcal{B}^\infty}\ip{f,g} ).
	$$
	for all $f,g\in \mathcal{S}(\R)$. Note that $J^{-1}$ also generates a cyclic subgroup of order four. Therefore the action $\Z_4\curvearrowright \mathcal{S}(\R)$, $n\mapsto (S_J)^n$ satisfies all the properties in Proposition \ref{prop:Combes}, and we conclude $\mathcal{A}_{\theta} \rtimes_\alpha \Z_4$ and $\mathcal{A}_{\frac{1}{\theta}} \rtimes_\alpha \Z_4$ are Morita equivalent (it does not matter which copy of $\Z_4$ we choose in $SL_2(\Z)$, as they are all conjugate to each other). For $k=6$, we take the action $\Z_6\curvearrowright \mathcal{S}(\R)$, $n\mapsto (S_H)^n$, and conclude in a similar fashion.
	
	$(1)\Longrightarrow(2)$: Looking at the rank of the $K_0$ group allows us to deduce $k=k'$. Since $\mathcal{A}_{\theta} \rtimes \Z_k$ has a unique trace and the image of $K_0$ under the tracial state has the form $\frac{1}{k}(\Z+\theta \Z)$, we get $\theta\sim_{Mob} \theta'$ in exactly the same way as in the proof of \cite[Theorem 4]{Rie81} (see the proof of Theorem \ref{thm:main_morita}, or Remark \ref{rem:alt_pf}). 
\end{proof}

Finally, we would like to explain another approach to prove Theorem \ref{thm:mortia_finite}, exploiting the fact that $\mathcal{A}_\theta\rtimes_\alpha F$ is an AF-algebra. This approach is suggested by Dominic Enders. We shall need the following preliminary result:
\begin{lem}\label{lem:Mobius}
	Suppose $\theta' = \frac{a\theta + b}{c\theta + d}$ for some $\begin{pmatrix}
	a & b \\
	c & d
	\end{pmatrix}\in GL_2(\Z)$. Then $\Z + \theta' \Z = \frac{1}{c\theta + d}(\Z + \theta \Z)$.
\end{lem}
\begin{proof}
	Let $m,n\in \Z$. Then
	\begin{align*}
		m+ \theta' n &= m + \left(  \frac{a\theta + b}{c\theta + d}  \right) n\\
		&= \frac{1}{c\theta + d}\left[  (dm+bn) + \theta (cm +an)   \right].
	\end{align*}
	This shows that $\Z+ \theta'\Z \subseteq \frac{1}{c\theta + d}(\Z + \theta \Z)$. On the other hand, we have
	\begin{align*}
		(am&-bn) + \theta' (-cm + dn) \\ &= (am-bn) + \left(  \frac{a\theta + b}{c\theta +d}  \right) (-cm + dn) \\
		&= \frac{1}{c\theta+d}\left[  acm\theta + adm - bcn\theta -bdn -acm\theta + adn\theta -bcm + bdn  \right] \\
		&= \frac{ad-bc}{c\theta + d}(m+ \theta n).
	\end{align*}
	Since $ad-bc = \pm 1$, we conclude that $\frac{1}{c\theta + d}(\Z+\theta Z)\subseteq \Z + \theta' \Z$. This completes the proof.
\end{proof}

	Now we explain the promised alternative proof of $(2)\Longrightarrow (1)$ in Theorem \ref{thm:mortia_finite}:
	For notational convenience we will write $\mathcal{C}=\mathcal{A}_\theta\rtimes_\alpha \Z_k$ and $\mathcal{D}=\mathcal{A}_{\theta'}\rtimes_\alpha \Z_{k}$. Using the Brown-Green-Rieffel Theorem \cite{BRG77} and the classification of AF algebras due to Elliott \cite{Ell76}, it is enough to exhibit an order isomorphism
	$(K_0(\mathcal{C}),K_0(\mathcal{C})^+)\rightarrow (K_0(\mathcal{D}),K_0(\mathcal{D})^+)$.	
	By assumption $\theta'=\frac{a\theta+b}{c\theta+d}$ for some $\begin{pmatrix}
	a & b \\
	c & d
	\end{pmatrix}\in GL_2(\Z)$ and hence $\Z + \theta' \Z = \frac{1}{c\theta + d}(\Z + \theta \Z)$ by Lemma \ref{lem:Mobius}. We may also assume that $c\theta+d>0$. Thus, we obtain an order isomorphism
	$$\varphi:\frac{1}{k}(\Z+\theta \Z)\rightarrow \frac{1}{k}(\Z+\theta'\Z),$$
	given by multiplication with $\frac{1}{c\theta+d}>0$.
	Let $\tau_\mathcal{C}$ and $\tau_\mathcal{D}$ denote the unique tracial states on $\mathcal{C}$ and $\mathcal{D}$, respectively. Consider the short exact sequence
	$$0\longrightarrow \ker((\tau_\mathcal{C})_*)\longrightarrow K_0(\mathcal{C})\stackrel{(\tau_\mathcal{C})_*}{\longrightarrow}\frac{1}{k}(\Z+\theta\Z)\longrightarrow 0.$$
	Note, that this sequence splits and since $K_0(\mathcal{C})\cong \Z^n$, where $n\in \lbrace 6,8,9,10\rbrace$ depending on $k$, we have an abstract isomorphism $\ker((\tau_\mathcal{C})_*)\cong \Z^{n-2}$. The same arguments show that $\ker((\tau_\mathcal{D})_*)\cong \Z^{n-2}$. From this we obtain an isomorphism $\psi:\ker((\tau_\mathcal{C})_*)\rightarrow \ker((\tau_\mathcal{D})_*)$. We conclude, that there exists an isomorphism $f:=\psi\oplus \varphi:K_0(\mathcal{C})\rightarrow K_0(\mathcal{D})$, such that the following diagram commutes:
	$$
	\begin{CD}
	K_0(\mathcal{C})@>(\tau_\mathcal{C})_*>>\frac{1}{k}(\Z+\theta\Z)\\
	@VVf V@VV\varphi V\\
	K_0(\mathcal{D})@>(\tau_\mathcal{D})_*>>\frac{1}{k}(\Z+\theta'\Z).
	\end{CD}
	$$
	We claim that $f$ is an order isomorphism. By commutativity of the diagram we have $(\tau_\mathcal{D})_*\circ f=\varphi\circ (\tau_\mathcal{C})_*=\frac{1}{c\theta+d}(\tau_\mathcal{C})_*$ and hence our claim follows from Proposition \ref{prop:order-iso}.

\bibliographystyle{plain}
\bibliography{Biblio-Database}

\end{document}